\documentclass[12pt,reqno]{amsart}
\usepackage{mathtools}
\usepackage{amssymb, mathabx, enumitem}
\usepackage{color}
\usepackage[colorlinks=true,linkcolor=blue,urlcolor=blue,citecolor=blue, pagebackref]{hyperref}
\usepackage{tikz-cd}
\usepackage{tikz}
\usepackage{tikz-3dplot}
\usepackage{graphics}
\usepackage{arydshln}
\usepackage{graphicx}%
\usepackage{multirow}%
\usepackage{amsmath,amssymb,amsfonts}%
\usepackage{amsthm}%
\usepackage{mathrsfs}%
\usepackage[title]{appendix}%
\usepackage{xcolor}%
\usepackage{textcomp}%
\usepackage{manyfoot}%
\usepackage{booktabs}%
\usepackage{algorithm}%
\usepackage{algorithmicx}%
\usepackage{algpseudocode}%
\usepackage{listings}%
\usepackage{tikz-cd}
\usepackage{mathtools}
\usepackage[numbers]{natbib}

\usepackage[margin=1in]{geometry}

\newcommand{\C}{\mathbb{C}}

\newcommand{\Z}{\mathbb{Z}}

\newcommand{\Gr}{\mathrm{Gr}}
\newcommand{\Gm}{\mathbb{G}_m}
\newcommand{\bT}{\mathbb{T}}
\newcommand{\Gmrot}{\Gm^{\mathrm{rot}}}
\newcommand{\cW}{W}
\newcommand{\cL}{\mathcal{L}}
\newcommand{\Spec}{\operatorname{Spec}}

\newcommand{\coh}{\mathrm{coh}}
\newcommand{\qcoh}{\mathrm{qcoh}}

\numberwithin{equation}{section}

\newtheorem{theorem}{Theorem}[section]

\newtheorem{remark}[theorem]{Remark}

\newtheorem{corollary}[theorem]{Corollary}

\newtheorem{proposition}[theorem]{Proposition}

\newtheorem{lemma}[theorem]{Lemma}

\newtheorem{definition}[theorem]{Definition}
\newtheorem{definition/lemma}[theorem]{Definition/Lemma}

\newtheorem{example}[theorem]{Example}
\newtheorem{notation}[theorem]{Notation}

\title{Derived Enhancements of T-Fixed Subschemes}
\author{Marc Besson}
\address{YMSC, Tsinghua University, Beijing, P.R. China, 100871}
\email{bessonm@tsinghua.edu.cn}
\author{Shiyixin Liang}
\address{YMSC, Tsinghua University, Beijing, P.R. China, 100871}
\email{liangsyx21@mails.tsinghua.edu.cn}

\begin{document}
\maketitle
\begin{abstract} 
For $X$ a conical affine symplectic singularity with $\mathbb{T}=T \times \mathbb{G}_m$-action, the fixed scheme $X^T$ and the map $X^T \rightarrow X$ carry much information about the geometry of $X$. In general, $X^T \rightarrow X$ fails to be a complete intersection. Thus, we study a derived intersection whose classical locus is the $T$-fixed subscheme $X^T$. We show that the structure of the symplectic singularity on $X$ produces a duality theorem for the structure sheaf of the derived intersection. The duality theorem allows us to study the structure of such derived intersections; in particular we describe their cohomological amplitude. 

An important source of symplectic singularities with $\mathbb{T}$-action are affine Grassmannian slices $\overline{W}^{\lambda}_{\mu}$. We pay particular attention to these slices when $G=\mathrm{SL}_{n+1}$, and we use the previously developed theory to characterize when $(\overline{W}^{\lambda}_{\mu})^T \rightarrow \overline{W}^{\lambda}_{\mu}$ is a complete intersection.
\end{abstract}

\section{Introduction}\label{Sec1}

\subsection{Conical symplectic singularities}
Conical symplectic singularities are emerging as an important class of spaces in geometric representation theory and mathematical physics. Let $X=\Spec(A)$ be an affine scheme with symplectic singularity. We always assume that $X$ carries an action of an algebraic torus $T \times \mathbb{G}_{m}=\mathbb{T}$, such that the $\mathbb{G}_m$ action repels $X$ from a point which will be denoted by $0$. We write $X^{\mathbb{T}}$ for the $\mathbb{T}$-fixed subscheme of $X$, and we  further assume that the $\mathbb{T}$-fixed subscheme $X^{\mathbb{T}}$ is $0$ with the reduced scheme structure. We write $X^T$ for the fixed subscheme of $X$ under the smaller torus $T$. In this work, we study the situation where $X^T$ is a possibly nonreduced scheme supported over the point $0=X^{\mathbb{T}}$. The nonreduced structure on $X^T$, as well as the embedding $X^T \rightarrow X$, carry important information about singularities of $X$; moreover the singularities of $X$ often have representation-theoretic significance.
In many cases of interest, 
\[X^T \rightarrow X\]
is not a complete intersection, or in other words, the ideal $I^T(X)$ cutting out $X^T$ in $X$ is not generated by a regular sequence. From the point of view of derived geometry, when a closed subscheme $Y \rightarrow X$ fails to be a complete intersection, there is a refined object one may consider, namely the derived intersection. This derived scheme has structure complex which can be represented by the Koszul complex of generators of $I^T(X)$.

In our setting of $X^T \rightarrow X$, we study $T$-homogeneous minimal generating sets $\{x_1, \dots x_n\}$ of $I^T(X)$, and study the attendent propertes of the Koszul complexes $K_A^{\bullet}(x_1, \dots x_n)$. In other words, we study the geometry of derived intersections $\mathbb{V}^{\bullet}(x_1, \dots x_n)$ whose classical truncation is $X^T$. 
Our main results are as follows. In section \ref{Sec3}, we first describe the minimal number of homogeneous generators. 
\begin{proposition}(Proposition \ref{prop:basis} in the text)
Any homogeneous basis of $(T^*_0X)^{\neq 0}$ may be lifted to a minimal homogeneous generating set of $I^T(X)$.
\end{proposition}
Here, since $0$ is fixed by $T$, the cotangent space $T^*_0X$ affords a representation of $T$ and $(T^*_0X)^{\neq 0}$ is the direct sum of the nontrivial weight subspaces. In light of this proposition, the cardinality $n$ of a minimal homogeneous generating set of $I^T(X)$ is equal to the dimension $\dim (T_0^*X)^{\neq 0}$.

 The symplectic singularities of $X$ impose a very strong relationship between the structure complex $\mathscr{O}_{\mathbb{V}^{\bullet}}$ and the dualizing complex $\omega_{\mathbb{V}^{\bullet}}$ of the derived intersection $\mathbb{V}^{\bullet}(x_1, \dots x_n)$. This leads to our second main result, which is a bigraded (homologically graded and $\mathbb{T}$-graded) Poincar\'e duality. We use the shorthand $\underline{x}=\{x_1, \dots x_n\}$ for a minimal homogeneous generating set, and we write $K_A^{\bullet}( \underline{x})$ for the Koszul complex on the chosen minimal generating set. Further, we write $d_X$ for $\dim X$ and write $\varpi_r$ for the standard character of $\mathbb{G}_m$; given an object $\mathscr{F} \in D_{\coh}^{\mathbb{T}}(X)$, we write $\mathscr{F}(\lambda+k\varpi_{r}):=\mathscr{F} \otimes_{\mathbb{C}} \mathbb{C}_{\lambda+k \varpi_r}$ for twisting by the $\mathbb{T}$-character $\lambda+k\varpi_r$. In section \ref{Sec5}, we prove the following duality result.

\begin{theorem} (Theorem \ref{thm2} and Corollary \ref{pcare} and in the text)

 We have $\mathbb{T}$-equivariant isomorphisms \[H^{-i}(K_A^{\bullet}(\underline{x}))^* \simeq H^{d_X -n+i}(K_A^{\bullet}(\underline{x})) \left(\mathrm{wt}(\mathrm{det}(T_0X)^{\neq 0})+ \frac{d_X}{2} \varpi_r \right)\] for all $i \in \mathbb{Z}$.
\end{theorem}

\subsection{Slices in Type A}
In the second part, we apply these results to produce a number of examples in a more specific setting. We consider $T \rightarrow G$ as the maximal torus of a reductive group, and we consider (spherical) affine Schubert varieties $\overline{\mathrm{Gr}}^{\lambda}_G$ associated to a cocharacter $\lambda$ of $T$. The $T$-fixed scheme $(\overline{\mathrm{Gr}}^{\lambda}_G )^T$ is of dimension $0$ and the components are indexed by a finite subset of cocharacters $\mu$. For each such $\mu$, one can define the slice $\overline{W}_{\mu}^{\lambda} \rightarrow \overline{\mathrm{Gr}}^{\lambda}_G$, transverse to $\overline{\mathrm{Gr}}^{\mu}$. These slices $\overline{W}^{\lambda}_{\mu}$ are conical affine symplectic singularities, carrying an action of $\mathbb{T}=T \times \mathbb{G}_m$. When $G=SL_{n+1}$, the descriptions of minimal number of generators and the structure of the derived intersections $\mathbb{V}^{\bullet}(I^T(X))$ are rather explicit and combinatorial. After some background material in Sections \ref{Sec6} and \ref{genminors}, we produce explicit formulas for minimal numbers of generators of $I^T(\overline{W}^{\lambda}_{0})$ as well as the cohomological amplitude of $\mathbb{V}^{\bullet}(I^T(\overline{W}^{\lambda}_0))$, in the setting where $G= SL_{n+1}$. In section \ref{examples} we present many examples. In section \ref{globaltheorems} we prove our last main result.
Let $\varpi_i^{\vee}$ be the fundamental coweights; so $\langle \varpi^{\vee}_i, \alpha_j \rangle = \delta_{ij}$.
\begin{theorem}(Theorem \ref{lcicriterion} in the text)
Let $\overline{\mathrm{Gr}}^{\lambda}$ be an affine Schubert variety for $G=SL_{n+1}$, with $\lambda$ a dominant coweight. Then 
\[(\overline{\mathrm{Gr}}^{\lambda})^T \rightarrow \overline{\mathrm{Gr}}^{\lambda}\] is a local complete intersection if and only if $\lambda = (n+1)  k \varpi_1^{\vee}$ or $ \lambda = (n+1)k \varpi_n^{\vee}$ for $k \in \mathbb{N}$.
\end{theorem}

\subsection{Symplectic Duality} 
One perspective on the $T$-fixed subschemes comes from symplectic duality. Symplectic singularities and symplectic resolutions play important roles in the study of geometric representation theory. A fascinating phenomenon in the theory of conical symplectic singularities is that they tend to come in dual pairs. 

One example is the Hikita conjecture. Let $X$ be a conical symplectic singularity in the sense of \cite{BeauSS} with a $\mathbb{T}=T \times \mathbb{G}_m$-action subject to some assumptions (see Section \ref{Sec2} for our precise list). Let $X^{!}$ denote its symplectic dual and suppose it admits a symplectic resolution $\widetilde{X}^!\to X^{!}$. Consider the $T$-fixed subscheme $X^T$ with respect to the Hamiltonian $T$-action on $X$. Hikita \cite{Hik15} conjectures an isomorphism of graded rings\footnote{To match the gradings of two sides, one need to use the twice of $\Gmrot$-grading on the left hand side.}
\[\C[X^T]\cong \mathsf{H}^\bullet_{\mathrm{sing}}(\widetilde{X}^!,\mathbb{C}),\]
where $\mathsf{H}_{\mathrm{sing}}^\bullet(-, \mathbb{C})$ denotes the singular cohomology of a topological space with coefficients in $\mathbb{C}$.
There are sevaral generalizations of this conjecture, for example the equivariant version \cite[Conjecture 8.9]{KTWWY} due to Nakajima, the quantum version \cite{KMP} due to Kamnitzer, McBreen and Proudfoot, and some refinements due to Hoang, Kyrlov and Matvieievskyi \cite{HKM}.

Our initial goal in beginning this project was to describe features on the symplectic dual side related to our derived intersection $\mathbb{V}^{\bullet}(I^T(X))$. While we were not successful in this regard, we have included some remarks in section \ref{remarks} in this direction.

\subsection{Acknowledgements} The major part of this work was done when both authors were at Tsinghua University, and this work is supported by National Key R\&D Program of China (No. 2025YFA1017400) and NSFC Grant 12225108 through the authors' advisor Peng Shan. The work was completed while the second author was in residence at SLMath for the Fall 2026 semester, supported by the National Science Foundation under Grant No. DMS-2424139.
We thank Peng Shan for her interest in this project, and for her advice and comments. The first author wishes to thank Federico Bongiorno, Sam Jeralds and Josh Kiers for comments and discussion.
The second author wishes to thank Vasily Krylov and Lin Chen for comments and discussion.

\section{Notation for Symplectic Singularities}\label{Sec2}

\subsection{Tori}
Let $T$ be an algebraic torus over $\mathbb{C}$. We have the coweight and weight lattices \[X_*(T):=\mathrm{Hom}(\mathbb{G}_m, T)\] and \[X^*(T):= \mathrm{Hom}(T, \mathbb{G}_m).\] These are dual lattices equipped with the perfect pairing \[\langle, \rangle: X_*(T) \times X^*(T) \rightarrow \mathrm{Hom}( \mathbb{G}_m, \mathbb{G}_m) = \mathbb{Z}\] 
\[(\mu, \lambda) \mapsto \lambda \circ \mu.\]
We will make much use of a larger torus \[\mathbb{T}:=T \times \Gmrot.\] We write $\varpi_r$ for the generator of $X^*(\Gmrot)$, corresponding to the standard representation of $\Gmrot$ on $\mathbb{A}^1$. Thus in general we write \[\lambda +s \varpi_r \in X^*(\mathbb{T})\]  for a weight of $\mathbb{T}$, where $\lambda \in X^*(T)$ and $s \in \mathbb{Z}$.
 
 \subsection{$H$-fixed subschemes}
The main reference for this subsection is \cite{DrinGm}. The original reference is \cite{FogFPS}. Let $H$ be an algebraic group over $\mathbb{C}$. Let $X$ be a $\mathbb{C}$-scheme with $H$-action. Let $h_X$ be the presheaf on $\mathbb{C}$-algebras associated to $X$, i.e.
\[h_X : \mathbb{C}\text{-}\mathrm{Alg} \rightarrow \mathrm{Set}\] 
\[R \mapsto X(R).\]
The definition of $H$-fixed subsheaf is as follows.

\begin{definition}
	For any $R\in \mathbb{C}\text{-}\mathrm{Alg}$, we view $\Spec (R)$ as an $H$-scheme with trivial $H$-action. Let $\mathrm{Hom}^H(\Spec R,X)$ be the set of $H$-equivariant morphisms from $\Spec R$ to $X$. This defines a presheaf
	\[h_X^H:\mathbb{C}\text{-}\mathrm{Alg}\to \mathrm{Set}\]
	by sending $R$ to $\mathrm{Hom}^H(\Spec (R),X)$. We have a natural map $h_X^H\to h_X$.
\end{definition}

We will focus on the case when $H$ is a torus. We first assume $X$ is affine, i.e.$: X=\Spec B$ for some $\mathbb{C}$-algebra $B$. An $H$-action on $X$ is equivalent to an $X^*(H)$-grading: 
$$B=\bigoplus_{\chi \in X^*(H)}B_{\chi}$$
on $B$, such that $B_{\chi_1}\cdot B_{\chi_2}\subset B_{\chi_1+\chi_2}$. Let $I$ be the ideal of $B$ generated by the homogeneous elements of nonzero degrees.
\begin{proposition}\label{representability}
	The presheaf $h_X^H$ is representable by $\Spec  B/I$.
\end{proposition}
\begin{proof}
	The set $\mathrm{Hom}^H(\Spec R,\Spec B)$ consists of ring homomorphisms $\phi:B\to R$ compatible with $H$-actions, which is equivalent to requiring those $B_{\chi}$ are sent to zero if $\chi\neq 0$. Thus $\mathrm{Hom}^H(\Spec R,\Spec B)=\mathrm{Hom}(\Spec  R, \Spec B/I)$.
\end{proof}

The following proposition assures the representability of $h_X^H$ for general $\mathbb{C}$-schemes $X$ of finite type. The proof is an easy generalization of \cite{DrinGm} (replacing $\mathbb{G}_m$ with the algebraic torus $H$).

\begin{proposition}[Proposition 1.2.2, \cite{DrinGm}]
	If $X$ is a $\mathbb{C}$-scheme $X$ of finite type, then $h_X^H$ is representable by a $\mathbb{C}$-scheme $X^H$ of finite type, and the natural morphism $X^H\rightarrow X$ is a closed embedding. 
\end{proposition}

\subsection{Symplectic Singularities}\label{symplecticsingassump}
We work in the following framework. Let $X$ be an affine symplectic singularity in the sense of Beauville \cite{BeauSS}: that is, the regular locus $X_{reg}$ is equipped with an algebraic symplectic 2-form $\sigma$ such that given any resolution of singularities \[\pi:\tilde{X} \rightarrow X,\] $\pi^*\sigma$ extends to an algebraic 2-form on $\tilde{X}$. We make the following additional assumptions.

\begin{itemize}
\item $X$ is a $\mathbb{T}$-variety
\item $\mathbb{C}[X]$ is positively graded under the $\Gmrot$ action with $\mathbb{C}[X]_0=\mathbb{C}$, so that $\Gmrot$ repels $X$ from the $\Gmrot-$fixed subscheme $0$. \footnote{One could consider the opposite action of $\Gmrot$ which contracts $X$ to $0$; under this action $A$ is negatively graded.}
\item $\sigma$ is $T$-invariant and $\Gmrot$ acts on $\sigma$ with weight 1.
\end{itemize}

\section{Ideals and Cotangent spaces}\label{Sec3}

Let $A$ be a Noetherian $\C$-algebra with a $\mathbb{T}$-action such that $A$ is positively graded with respect to $\Gmrot$ and such that $A^{\Gmrot}=A_0=\mathbb{C}$. Thus we can write \[A = \bigoplus_{\lambda + k \varpi_r} A_{\lambda+k \varpi_r} \] with $k \geq 0$ and $\lambda \in X^*(T)$; we have $A_{\lambda}=0$ for $\lambda \neq 0$. We have the maximal ideal \[ \mathfrak{m}_A:= \bigoplus_{\lambda \in X^*(T), k >0} A_{\lambda+k \varpi_r}.\]

Let $I^T(A)$ denote the ideal generated by the non-trivial $T$-weight spaces. This is a finitely generated ideal by our Noetherian assumptions. We have $I^T(A) \subset \mathfrak{m}_A$. We wish to compare a generating set of the ideal $I^T(A)$ to the cotangent space $\mathfrak{m}_A/\mathfrak{m}_A^2$. We write $(\mathfrak{m}_A/\mathfrak{m}_A^2)^T$ for the $T$-invariants and $(\mathfrak{m}_A/\mathfrak{m}_A^2)^{\neq 0}$ for the direct sum of the nontrivial $T$-weight spaces.

\begin{proposition}\label{prop:generator}
	The natural map \[ I^T(A)/\mathfrak{m}_A \cdot I^T(A) \rightarrow \mathfrak{m}_A / \mathfrak{m}_A^2\] yields an isomorphism \[I^T(A)/\mathfrak{m}_A \cdot I^T(A) \simeq (\mathfrak{m}_A / \mathfrak{m}_A^2)^{\neq 0}.\]
\end{proposition}

\begin{proof}
	We first show injectivity. Let $\overline{x} \in I^T(A)/\mathfrak{m}_A \cdot I^T(A)$ be a homogeneous element of $T$-weight $\lambda$ such that the image of $\overline{x}$ in $\mathfrak{m}_A/\mathfrak{m}_A^2$ is zero. Then choosing a homogeneous representative $x \in I^T(A)$ we have $x \in \mathfrak{m}^2_A$. If $\lambda =0$, then since $x \in I^T(A)$, we may write $x= \sum y_i b_i$ for $y_i, b_i$ homogeneous; here $b_i$ has nontrivial $T$-weight and $\operatorname{wt}(y_i b_i)=\operatorname{wt}(x)=0$. Then $\operatorname{wt}(y_i)=-\operatorname{wt}(b_i)\neq 0$, so $y_i\in \mathfrak{m}_A$ by our assumptions. Thus $x \in \mathfrak{m}_A \cdot I^T(A)$. If $\lambda \neq 0$ then since $x \in \mathfrak{m}_A^2$, we may write $x = \sum_i y_{i} x_{i}$ for homogeneous $y_i,x_i\in \mathfrak{m}_A$ such that $\operatorname{wt}{(y_ix_i)}=\operatorname{wt}(x)$. Then either $y_{i}$ or $x_i$ has a nontrivial $T$-weight for each $i$, so $x \in \mathfrak{m}_A \cdot  I^T(A)$. 
	
	For surjectivity we observe that the inclusion \[I^T(A) \rightarrow \mathfrak{m}_A \] induces a surjection of graded $\mathbb{C}$-vector spaces \[I^T(A) \rightarrow \mathfrak{m}_{A}^{\neq 0} \rightarrow 0\] and thus a surjection \[ I^T(A) \rightarrow (\mathfrak{m}_A/\mathfrak{m}_A^2)^{\neq 0} \rightarrow 0.\] This latter map factors through $I^T(A) / \mathfrak{m}_A \cdot I^T(A)$ and we have our surjectivity.
\end{proof}

\begin{proposition}\label{prop:basis}
	A homogeneous basis $\{\overline{x}_{\lambda_{1}+j_1\varpi_r}, \dots \overline{x}_{\lambda_{n}+j_n \varpi_r }\}$ for $(\mathfrak{m}_A/\mathfrak{m}_A^2)^{\neq 0}$ may be lifted to a homogeneous minimal generating set $\{x_{\lambda_1+j_1 \varpi_r}, \dots x_{\lambda_n +j_n \varpi_r} \} $ for $I^T(A)$. In particular, the minimal number of generators of $I^T(A)$ is $\dim_{\mathbb{C}} (\mathfrak{m}_A/\mathfrak{m}_A^2)^{\neq 0}$.
\end{proposition}

\begin{proof}
	This is an application of Nakayama's lemma in the positively graded case.
\end{proof}

Now we discuss the relations between two minimal generating sets.
\begin{lemma}\label{changeofbasis}
	Let $\{x_1, \dots x_n\}$ and $\{y_1, \dots y_n\}$ be two homogeneous minimal generating sets of $I^T(A)$. Then there exists an invertible matrix $\phi$ with coefficients in $A$ such that $\phi(\{x_1, \dots x_n\})=\{y_1, \dots y_n\}.$
\end{lemma}
\begin{proof}
	Since $\{x_1,\dots,x_n\}$ is a generating set of $I^T(A)$, we can find homogeneous elements $a_{ij}\in A$ for $i,j=1,2,\dots,n$ such that 
	\[y_i=\sum_{i=1}^n a_{ij}x_j.\]
	We can choose $a_{ij}$ such that $\operatorname{wt}(a_{ij}x_j)=\operatorname{wt}(y_i)$, so $\operatorname{wt}(a_{ij})=\operatorname{wt}(y_i)-\operatorname{wt}(x_j)$. Define $\phi:=(a_{ij})$. The determinant of $\phi$ is
	\[\det\left(\phi\right)=\sum_{\sigma\in \mathfrak{S}_n}\operatorname{sgn}(\sigma) \prod_{i=1}^n a_{i\sigma(i)},\]
	and
	\[\operatorname{wt}\left(\prod_{i=1}^n a_{i\sigma(i)}\right)=\sum_{i=1}^n \operatorname{wt}(a_{i\sigma(i)})=\sum_{i=1}^n\big(\operatorname{wt}(y_i)-\operatorname{wt}(x_{\sigma(i)})\big)=0\]
	because $\{x_1,\dots,x_n\}$ and $\{y_1,\dots,y_n\}$ have the same set of weights. Then we have
	\[\operatorname{wt}(\det\left(\phi\right))=0,\]
	so by our assumptions, $\det\left(\phi\right)\in A_0=\C$. This is a nonzero constant because modulo $\mathfrak{m}_A$, $(\overline{a_{ij}})$ determines an automorphism of the vector space $I^T(A)/\mathfrak{m}_A\cdot I^T(A)$ sending $\{\overline{x_1},\dots,\overline{x_n}\}$ to $\{\overline{y_1},\dots,\overline{y_n}\}$. In conclusion, $\det\left(\phi\right)$ is an invertible element in $A$, thus $\phi$ is an invertible matrix with coefficients in $A$.
\end{proof}

\section{Derived Enhancements of T-fixed Subschemes}\label{Sec4}
Going forward, $A$ will refer to a $\mathbb{C}$-algebra such that $\Spec (A)$ is a conical affine symplectic singularity; in contrast, $R$ will denote an arbitrary $\mathbb{C}$-algebra. Given an arbitrary sequence of elements $(x_1, \dots x_n)$ in a $\mathbb{C}$-algebra of finite type $R$, one way to construct a derived scheme whose classical truncation is $\Spec (R/(x_1, \dots x_n))$ is to consider the Koszul complex $K_R^{\bullet}(x_1, \dots x_n)$, viewed as the commutative differential graded algebra (CDGA) of derived functions on the derived intersection $\mathbb{V}^{\bullet}(x_1 \dots x_n)$. This derived scheme is the pullback in the ($\infty-$)category of derived schemes 

\begin{center}
\begin{tikzcd}
\mathbb{V}^{\bullet}(x_1, \dots x_n) \arrow[r] \arrow[d] &\Spec  R \arrow[d] \\
\Spec  \mathbb{C} \arrow["0",r] & \mathbb{A}^n\\
\end{tikzcd}
\end{center}
where the map $\Spec (R) \rightarrow \mathbb{A}^n$ is determined by the elements $x_1, \dots x_n$.
We wish to investigate the geometry of the derived intersections defined by a generating set of $I^T(A)$.

Let $X$ be an affine symplectic singularity with $\mathbb{T}$-action satisfying the conditions of section \ref{symplecticsingassump}. We write $A$ for $\mathbb{C}[X]$. We fix a homogeneous generating set $\{x_1, \dots x_n\}$ of $I^T(A)$, and write $\underline{x}$ for this sequence. As our focus going forward is more geometric, we often write $I^T(X)$ for $I^T(A)$.
\begin{definition}

Let
\[ X \rightarrow \mathbb{A}^{n}\] be the map of schemes induced by 
\[\mathbb{C}[T_1, \dots T_n] \rightarrow A \]
\[T_i \mapsto x_i.\]

Then $\mathbb{V}^{\bullet}(I^T(X),\underline{x})$ is defined to be the following fiber product in the $(\infty-)$category of derived schemes:

\begin{equation}
\begin{tikzcd}
\mathbb{V}^{\bullet}(I^T(X), \underline{x})\arrow[r] \arrow[d] & X \arrow[d] \\
\Spec  \mathbb{C} \arrow["0",r] & \mathbb{A}^{n} \\
\end{tikzcd}
\end{equation}
\end{definition}

Having fixed $X$ and a homogeneous minimal generating set $\underline{x}$ of $I^T(X)$ as above, we sometimes write $\mathbb{V}^{\bullet}(I^T(X))$ or even $\mathbb{V}^{\bullet}$ for $\mathbb{V}^{\bullet}(I^T(X), \underline{x})$.
This is a derived intersection of the vanishing loci of the $x_i$. The structure complex $\mathscr{O}_{\mathbb{V}^{\bullet}}$ is represented by the Koszul complex $K^{\bullet}_A(x_1, \dots x_n)$. In particular, $H^0(K^{\bullet}_A(x_1, \dots x_n))=\mathbb{C}[X^T]$.

\begin{lemma}
Given any two minimal generating sets $\{x_1, \dots x_n\}$ and $\{y_1, \dots y_n\}$ of $I^T(X)$ there exists an isomorphism of derived intersections: 

\[\mathbb{V}^{\bullet}(I^T(X), \underline{x}) \simeq \mathbb{V}^{\bullet}(I^T(X), \underline{y}).\]
\end{lemma}

\begin{proof}
We have an element $\phi$ of $GL(A)$ which sends the first generating set of $I^T(A)$ to the second by Lemma \ref{changeofbasis}. Such transformations induce isomorphisms (as differential graded algebras) of the Koszul complexes \[\phi: K^{\bullet}_A(x_1, \dots x_n) \simeq K^{\bullet}_A(y_1, \dots y_n),\] see \cite[15.28.4]{stacks-project}.
\end{proof}

\subsection{Relation between $X^T, \mathbb{V}^{\bullet}(I^T(X), \underline{x}), X^{T, \bullet}$}
We mention here that there are more sophisticated derived schemes that one may consider. The derived intersection $\mathbb{V}^{\bullet}(I^T(X),\underline{x})$ has $X^T$ as its classical truncation, but as a derived scheme it has a nontrivial $T$-action. For an easy example of this phenomenon, see Example \ref{sl3example} and Remark \ref{sl3remark}. Comparing $K^{\bullet}_A(x_1, \dots x_n)$ with its 0th cohomology $H^0(K_A^{\bullet}(x_1, \dots x_n))= \mathbb{C}[X^T]$, we see that we have killed the generators with nontrivial $T$-weight in degree 0 by ``attaching cells'' $e_i$ in degree $-1$ such that $de_i=x_i$. However there may be relations between the $e_i$ which have nontrivial $T$-weight; these will contribute to nontrivial $T$-actions on $H^{-1}$ of the Koszul complex. We could kill these by additional attachments of cells; this would be (a variant of) the process of Tate \cite{Tate}. Presumably, in the limit, we would produce a derived scheme $X^{T, \bullet}$ which would deserve to be called the derived $T$-fixed subscheme. It is not unreasonable to consider $\mathbb{V}^{\bullet}(I^T(X), \underline{x})$ as a first derived approximation to $X^{T, \bullet}$. This latter object is certainly worthy of study, but in this paper we content ourselves with the derived intersection $\mathbb{V}^{\bullet}$, which already appears in character formulas, see Section \ref{K-Theory}.

\section{Bigraded Poincare duality}\label{Sec5}
\subsection{Reminders on dualizing complexes}

For a derived scheme $Y$ with a $\mathbb{T}$-action, we will write $D^{\mathbb{T}}_{\qcoh}(Y)$ for the stable $\infty$-category of quasi-coherent sheaves on the quotient derived stack $[Y/\mathbb{T}]$; we write $D^{\mathbb{T}}_{\coh}(Y)$ for the sub-$\infty$-category consisting of bounded objects with coherent cohomologies. When $Y$ is a classical scheme, these categories can be described using the usual equivariant derived categories. 

We will recall some knowledge on $!$-pullback and dualizing complexes as established in \cite[Sec. 6.4]{LurieSAG}. Let $Y,Z$ be derived schemes of finite type over $\Spec \C$ with $\mathbb{T}$-actions, and let $f:Y \to Z$ be a $\mathbb{T}$-equivariant proper
and locally of finite Tor-amplitude morphism; by abuse of notation, we also denote the induced morphism between the quotient derived stacks by $f$. Then one can define the $!$-pull back functor, which serves as a right adjoint of $f_*$:
\[f^!:D^\mathbb{T}_{\qcoh}(Z) \to D^\mathbb{T}_{\qcoh}(Y),\]
see \cite[Prop. 6.4.2.1]{LurieSAG}. We also need the following compatibility between $!$-pullback and $*$-pullback.
\begin{proposition}[{\cite[Prop. 6.4.2.1]{LurieSAG}}]\label{luriesquare}
	Given a morphism $f:Y \to Z$ satisfying above conditions, and a Cartesian diagram
	\begin{center}
		\begin{tikzcd}
			Y' \arrow["{g'}",r] \arrow["{f'}",d] & Y \arrow["f",d] \\
			Z' \arrow["g",r] & Z, \\
		\end{tikzcd}
	\end{center} 
	we have the associated diagram
	\begin{center}
		\begin{tikzcd}
			D^\mathbb{T}_{\qcoh}(Z) \arrow["{g^*}",r] \arrow["{f^!}",d] & D^\mathbb{T}_{\qcoh}(Z') \arrow["f'^!",d] \\
			D^\mathbb{T}_{\qcoh}(Y) \arrow["g'^*",r] & D^\mathbb{T}_{\qcoh}(Y') \\
		\end{tikzcd}
	\end{center} 
	and moreover the canonical map $g'^*f^!\to f'^!g^*$ is an isomorphism.
\end{proposition}

When $Z=\mathrm{pt}=\Spec \mathbb{C}$, we write $\omega_Y:=f^!(\mathscr{O}_{\mathrm{pt}})$. More generally, if $Y\to \mathrm{pt}$ may not be proper but admits an open embedding into a proper derived scheme $\overline{Y}$ over $\mathrm{pt}$, we also define $\omega_Y:=j^*\omega_{\overline{Y}}$ for $j:Y \hookrightarrow \overline{Y}$; $\omega_Y$ doesn't depend on the compactification $j:Y \rightarrow \overline{Y}$. If $Y$ is a smooth scheme of finite type over $\Spec \C$, then $$\omega_Y\simeq \bigwedge^{\dim_Y}\Omega_Y^1[\dim_Y],$$ where $\Omega_Y^1$ is the sheaf of $1$-forms on $Y$.

According to \cite[Prop. 6.6.3.1]{LurieSAG}, $\omega_Y$ is a dualizing complex in the sense of \cite[Def. 6.6.1.1]{LurieSAG}; one can define the Grothendieck-Serre duality functor
\[\mathbb{D}_Y:=\underline{\mathrm{Hom}}(-,\omega_Y): D^\mathbb{T}_{\coh}(Y)^{\mathrm{op}} \to D^\mathbb{T}_{\coh}(Y).\]
For a proper and locally of finite Tor-amplitude morphism $f:Y \to Z$, we have natural isomorphisms
\[f_*\mathbb{D}_Y\simeq f_*\underline{\mathrm{Hom}}(-,f^!\omega_Z)\simeq \underline{\mathrm{Hom}}(f_*(-),\omega_Z)\simeq \mathbb{D}_Z f_*.\]

\subsection{Dualizing complexes of the derived enhancements}

We continue to work in the same framework as Section 2, so $X$ is a conical symplectic singularity with $\mathbb{T}$-action subject to our additional assumptions, and $\underline{x}$ is a fixed minimal homogeneous generating set of $I^T(X)$. In this section we study the dualizing complex $\omega_{\mathbb{V}^{\bullet}}$.

\begin{notation}\label{notation}
Given an object $\mathscr{F} \in D_{\coh}^{\mathbb{T}}(X)$ we may twist the $\mathbb{T}$-equivariant structure by a character $\lambda + k \varpi_r \in X^*(\mathbb{T})$; we write $\mathscr{F}(\lambda+k\varpi_{r}):=\mathscr{F} \otimes_{\mathbb{C}} \mathbb{C}_{\lambda+k \varpi_r}$ for this object.
\end{notation}

\begin{proposition}\label{dualizing}
Let $X$ satisfy the assumptions above. We have an isomorphism 

\[\omega_X \simeq \mathscr{O}_X \Bigl(\frac{d_X}{2}\varpi_{r}\Bigr) [d_X]\] in $D^{\mathbb{T}}_{\coh}(X)$. 
\end{proposition}

\begin{proof}
By Prop 1.3 in \cite{BeauSS}, X is rational Gorenstein, so we have an isomorphism $\omega_X \simeq \mathscr{L}[d_X]$ in $D^{\mathbb{T}}_{\coh}(X)$ for some invertible sheaf $\mathscr{L}$. We write $j: X_{reg} \rightarrow X$; since $X_{reg}$ is regular we have an isomorphism

 \[j^*(\mathscr{L}) \simeq \bigwedge^{d_X} \Omega^1_{X_{reg}}.\] 
 
 Recall that $\sigma$ is the symplectic form over $X_{reg}$. Since $\sigma$ is nondegenerate with $\mathbb{T}$-weight $ \varpi_r$, the section $\sigma^{\frac{1}{2} d_X}$ is a nonvanishing form on $X_{reg}$ of top degree with $\mathbb{T}$ weight $ \frac{d_X}{2} \varpi_{r}$. We can use $\sigma^{\frac{1}{2}d_X}$ to trivialize $\bigwedge^{d_X} \Omega^1_{X_{reg}} :$ 
 
\[\sigma^{\frac{1}{2}d_X}: \bigwedge^{d_X} \Omega^1_{X_{reg}} \simeq  \mathscr{O}_{X_{reg}}\Bigl(\frac{d_X}{2}\varpi_{r}\Bigr) .\] 

Since $X$ is a symplectic singularity and thus normal, by algebraic Hartogs' lemma we have $R^0j_*\mathscr{O}_{X_{reg}} \simeq \mathscr{O}_X$. Lastly we use the projection formula for $j$; 

\[\mathscr{L} \simeq \mathscr{L} \otimes_{\mathscr{O}_X} R^0j_* \mathscr{O}_{X_{reg}} \simeq R^0j_* (j^* \mathscr{L} \otimes_{\mathscr{O}_{X_{reg}}} \mathscr{O}_{X_{reg}})=R^0j_* \Bigl(\mathscr{O}_{X_{reg}}\Bigl(\frac{d_X}{2}\varpi_{r}\Bigr)  \Bigr) \simeq \mathscr{O}_{X}\Bigl(\frac{d_X}{2}\varpi_{r}\Bigr) . \] 

We conclude that $\omega_X \simeq \mathscr{O}_X\bigl(\frac{d_X}{2}\varpi_{r}\bigr) [d_X]$.
\end{proof}

We now introduce our derived intersection into the picture. 
As above, $\mathbb{V}^{\bullet}(I^T(X), \underline{x})$ fits in the following Cartesian diagram:

\begin{center}
\begin{equation}\label{cartesiandiagram}
\begin{tikzcd}
\mathbb{V}^{\bullet}(I^T(X), \underline{x}) \arrow["i",r] \arrow["g'",d] & X \arrow["g",d] \\
\Spec  \mathbb{C} \arrow["i_0",r] & \mathbb{A}^n \\
\end{tikzcd}
\end{equation}
\end{center}
Note that $\mathbb{A}^n$ is $\mathbb{T}$-equivariantly isomorphic to $(T_0X)^{\neq 0}$ by Proposition \ref{prop:generator}.
The following result shows that the symplectic singularities of $X$ (as well as the torus action) have strong implications for the dualizing complex of the derived scheme $\mathbb{V}^{\bullet}(I^T(X), \underline{x})$.

\begin{proposition}\label{prop:dualizingcomplex}
We have an isomorphism in $D^{\mathbb{T}}_{\coh}\left(\mathbb{V}^{\bullet}(I^T(X), \underline{x})\right)$: \[\omega_{\mathbb{V}^{\bullet}} \simeq \mathscr{O}_{\mathbb{V}^{\bullet}}\Bigl(\mathrm{wt}\left(\mathrm{det} ((T_0X)^{\neq 0})\right) +\frac{d_X}{2}\varpi_r\Bigr)[d_X-n].\]
\end{proposition}

\begin{proof}
The closed immersion $ \Spec  \mathbb{C} \rightarrow \mathbb{A}^n=(T_0X)^{\neq 0}$ is proper and has locally finite tor-dimension. Then by Proposition \ref{luriesquare}, the canonical map $g^{'*} i_0^{!} \rightarrow i^!g^*$ is an equivalence, where these maps come from diagram \eqref{cartesiandiagram}.

We have 

\begin{align*}
\omega_{\mathbb{V}^{\bullet}} & \simeq i^! \omega_{X} & \\
& \simeq i^! \mathscr{O}_X \Bigl(\frac{d_X}{2}\varpi_r\Bigr)[d_X] & \text{by Proposition \ref{dualizing}} \\
& \simeq i^! g^* \mathscr{O}_{\mathbb{A}^n}\Bigl(\frac{d_X}{2}\varpi_r\Bigr)[d_X] & \\
& \simeq (g')^* i_0^! \mathscr{O}_{\mathbb{A}^n} \Bigl(\frac{d_X}{2}\varpi_r\Bigr)[d_X]  & \text{by Proposition \ref{luriesquare}.} \\
\end{align*}
We have a $\mathbb{T}$-equivariant isomorphism of dualizing complexes 
\[\omega_{\mathbb{A}^n} \simeq \bigwedge^n \Omega^1_{\mathbb{A}^n}[d_{\mathbb{A}^n}] \simeq \mathscr{O}_{\mathbb{A}^n}\left(\mathrm{wt}(\mathrm{det} (\mathbb{A}^{n})^{*})\right)[d_{\mathbb{A}^n}].\]

Substituting this expression in the above, we obtain 

\begin{align*}
\omega_{\mathbb{V}^{\bullet}} & \simeq (g')^* i_0^! \mathscr{O}_{\mathbb{A}^n} \Bigl(\frac{d_X}{2}\varpi_r\Bigr)[d_X]  \\
& \simeq (g')^* i_0^! \omega_{\mathbb{A}^n}\Bigl(\mathrm{wt}(\det(\mathbb{A}^n))+\frac{d_X}{2}\varpi_r\Bigr)[d_X-d_{\mathbb{A}^n}] \\
& \simeq (g')^* \omega_{pt}\Bigl(\mathrm{wt}(\det(\mathbb{A}^n))+\frac{d_X}{2}\varpi_r\Bigr)[d_X-d_{\mathbb{A}^n}] \\
& \simeq (g')^* \mathscr{O}_{pt}\Bigl(\mathrm{wt}(\det(\mathbb{A}^n))+\frac{d_X}{2}\varpi_r\Bigr)[d_X-d_{\mathbb{A}^n}] \\
& \simeq \mathscr{O}_{\mathbb{V}^{\bullet}}\Bigl(\mathrm{wt}(\det(\mathbb{A}^n))+\frac{d_X}{2}\varpi_r\Bigr)[d_X-d_{\mathbb{A}^n}]. \\
\end{align*}
Finally we may make the $\mathbb{T}$-equivariant identification $\mathbb{A}^n=(T_0X)^{\neq 0}$ to obtain 
\[\omega_{\mathbb{V}^{\bullet}} \simeq \mathscr{O}_{\mathbb{V}^{\bullet}}\Bigl(\mathrm{wt}\left(\det((T_0X)^{\neq 0})\right)+\frac{d_X}{2}\varpi_r\Bigr)[d_X-n].\]
\end{proof}

\subsection{Bigraded Poincar\'e duality}
Applying the results of the previous section to our specific case, we obtain the following corollary.

\begin{theorem}\label{thm2}
Let \[g': \mathbb{V}^{\bullet}(I^T(X), \underline{x}) \rightarrow \Spec  \mathbb{C}\] be the structural morphism of diagram 6.1. We have isomorphisms in $D^{\mathbb{T}}_{\coh}(pt)$:

\begin{equation}\label{eq:isom_thm}
	\mathbb{D}_{pt}g'_* \mathscr{O}_{\mathbb{V}^{\bullet}} \simeq g'_* \mathscr{O}_{\mathbb{V}^{\bullet}}\Bigl(\mathrm{wt}\left(\mathrm{det}(T_0X^{\neq 0})\right)+\frac{d_X}{2}\varpi_r\Bigr)[d_{X}-n].
\end{equation}
\end{theorem}

\begin{proof}
Note that \[g': \mathbb{V}^{\bullet}(I^T(X), \underline{x}) \rightarrow \Spec  \mathbb{C}\] is finite and thus proper. So we have an isomorphism \[  g'_* \mathbb{D}_{\mathbb{V}^{\bullet}} \simeq \mathbb{D}_{pt}g'_*.  \] Then \[\mathbb{D}_{pt} g'_* \mathscr{O}_{\mathbb{V}^{\bullet}} \simeq g'_* \mathbb{D}_{\mathbb{V}^{\bullet}} \mathscr{O}_{\mathbb{V}^{\bullet}} \simeq g'_* \omega_{\mathbb{V}^{\bullet}} \simeq g'_*\mathscr{O}_{\mathbb{V}^{\bullet}}\Bigl(\mathrm{wt}\left(\mathrm{det}(T_0X^{\neq 0})\right)+\frac{d_X}{2}\varpi_r\Bigr)[d_{X}-n].\] 
\end{proof}

\begin{corollary}\label{pcare}
We have a $\mathbb{T}$-equivariant isomorphism 
\begin{equation*}\label{eq:isom}
	H^{-i}(K_A^{\bullet}(\underline{x}))^* \simeq H^{d_X -n+i}(K_A^{\bullet}(\underline{x}))\Bigl(\mathrm{wt}\left(\mathrm{det}(T_0X^{\neq 0})\right)+\frac{d_X}{2}\varpi_r\Bigr).
\end{equation*}
%Moreover, this is an isomorphism of $H^0 (K_A^\bullet (\underline{x}))=A/I^T(X)=\C[X^T]$-modules.
\end{corollary}
\begin{remark}
	Moreover, $g'_*\mathscr{O}_{\mathbb{V}^\bullet}$ can be equipped with a $K_A^\bullet (\underline{x})$-module structure, and $\mathbb{D}_{\mathrm{pt}} g'_*\mathscr{O}_{\mathbb{V}^\bullet}$ is the dual $K_A^\bullet (\underline{x})$-module. The isomorphism \eqref{eq:isom_thm} is compatible with the $K_A^\bullet (\underline{x})$-module structure by Proposition \ref{prop:dualizingcomplex}. In particular, \eqref{eq:isom} is an isomorphism of $H^0 (K_A^\bullet (\underline{x}))=\C[X^T]$-modules.
\end{remark}

After Corollary \ref{pcare}, we have the following bounds on the cohomology of the structure complex (Koszul complex):

\begin{corollary}\label{amplitude}
 The cohomology $H^{i}(\mathscr{O}_{\mathbb{V}^{\bullet}})=H^i(K_A^{\bullet}(\underline{x}))=0$ for $i \notin [d_X-n, 0]$, where $n$ is the dimension of $(T_0X)^{\neq 0}$. %Moreover, \[H^{{d_{\overline{\mathrm{Gr}}^{\lambda}}-d_{\mathbb{A}^n}}}(K^{\bullet}(f_{\lambda_1+k_1 \varpi_r} \dots f_{\lambda_n+k_n \varpi_r}))\] is dual (as $T \times \mathbb{G}_m$-modules) to $\mathbb{C}[(\overline{\mathrm{Gr}}_0^{\lambda})^T]$.
\end{corollary}

We wish to highlight a further corollary. We say a polynomial $P(T) \in \mathbb{C}[T]$ is palindromic if there exists an $k \in \mathbb{Z}^{\geq 0}$ such that $T^k P(T^{-1})=P(T)$.

\begin{corollary}\label{sympoly}
Let $X$ satisfy $\dim X=\dim  (T_0X)^{\neq 0}$. Then the Poincar\'e polynomial of $\mathbb{C}[X^{T}]$ is palindromic.
\end{corollary}

\begin{proof}
We apply the general Poincar\'e duality of \ref{pcare} in the case when the cohomology is concentrated in degree 0.
\end{proof}

\section{Additional notation for Slices}\label{Sec6}
Our aim in the following sections is to apply the results of Sections \ref{Sec4} and \ref{Sec5} to an interesting class of symplectic singularities. Our preferred symplectic singularities are the slices $\overline{W}^{\lambda}_{\mu}$ associated to spherical affine Schubert varieties of type $A,D,E$. For technical reasons we mostly restrict our attention to type $A$. 

In this section we supply some additional background and notation for the slices $\overline{W}^{\lambda}_{\mu}$.
\subsection{Reductive groups}
We use the standard notation for reductive groups. We let $G$ be a reductive group over $\mathbb{C}$, and fix $T \subset B \subset G$ a maximal torus and Borel. We write $\Phi$ for $\Phi(G,T)$ the root system, and $\Phi^{\vee}$ for the coroots. The positive roots (weights of the $T$-action on $\mathrm{Lie}(B)$) are denoted $\Phi^+$ and a base of simple roots is denoted $\alpha_1, \dots \alpha_r$. Simple coroots are denoted $\alpha_i^{\vee}$.
\subsection{Affine Grassmannians}The following material is standard and may be found in many references, for instance \cite{ZhuGrass}.
For $G$ an algebraic group, the functor \[LG: \mathbb{C}\text{-}\mathrm{Alg} \rightarrow \mathrm{Set}\] \[R \mapsto G(R((t))),\] is representable by an ind-scheme. We also have the arc group $L^+G$ and congruence subgroup $L^{--}_{\mathrm{pol}}G$ which represent \[ R \mapsto G(R[[t]])\] and \[R \mapsto  \mathrm{ker}_{\mathrm{ev}_{0}(t^{-1})}\left( G(R[t^{-1}]) \rightarrow G(R)\right).\] These are represented by a scheme of infinite type over $\mathbb{C}$ and an ind-scheme over $\mathbb{C}$, respectively. We will also make passing use of the functor 
\[L^{--}G: R \mapsto  \mathrm{ker}_{\mathrm{ev}_{0}(t^{-1})}\left( G(R[[t^{-1}]]) \rightarrow G(R)\right).\] It is represented by an affine scheme of infinite type.

We also have the affine Grassmannian $\mathrm{Gr}_G$, which represents the sheafification in the etale topology of \[R \mapsto LG(R)/L^+G(R).\] This is representable by an ind-projective ind-scheme, for $G$ reductive. Maps of affine algebraic groups \[H \rightarrow G\] induce \[LH \rightarrow LG,\] \[L^+H \rightarrow L^+G,\] \[\mathrm{Gr}_H \rightarrow \mathrm{Gr}_G,\] and so on.

We pick a uniformizer $t$ such that $\mathbb{G}_{m} = \Spec \mathbb{C}[t,t^{-1}]$.
Any coweight $\lambda \in X_*(T)=\mathrm{Hom}(\mathbb{G}_m, T)$ may be composed with the inclusion $T \rightarrow G$ to yield a map $\mathbb{G}_m \rightarrow G$, and as usual we write $t^{\lambda}$ for $\lambda(t) \in LG(\mathbb{C})$. By abuse of notation we also write $t^{\lambda}$ for the image of this point in $\mathrm{Gr}_G(\mathbb{C})$. Note that $t^{\lambda} \in LG$ depends on the choice of uniformizer $t$, but $t^{\lambda} \in \mathrm{Gr}_G$ does not.

The affine Grassmannian is a homogeneous space for $LG$. There is an additional action of $\mathbb{G}_m$ on $\mathrm{Gr}_G$, coming from rescaling the uniformizing parameter (called loop rotation); we write $\Gmrot$ for this copy of $\mathbb{G}_{m}$. Most relevant in this paper will be a restricted action; we restrict \[(LG \rtimes \Gmrot ) \times \mathrm{Gr}_G \rightarrow \mathrm{Gr}_G\]  to \[(T \times \Gmrot )\times \mathrm{Gr}_G=\mathbb{T} \times \mathrm{Gr}_G  \rightarrow \mathrm{Gr}_G.\]

We have isomorphisms $\pi_0(Gr_G) \simeq \pi_1^{alg}(G)$. For instance we have \[\pi_0(\mathrm{Gr}_{SL_{n+1}})=1 \] and for $T$ a maximal torus in $SL_{n+1}$ we have \[\pi_0 (\mathrm{Gr}_T) \simeq \mathbb{Z}\Phi^{\vee}.\]

\subsection{Affine Schubert Varieties}
The Cartan decomposition yields \[\mathrm{Gr}_G = \bigsqcup_{\lambda \in X_*(T)^+} L^+G t^{\lambda} L^+G/L^+G.\]  The orbits 
\[\mathrm{Gr}_G^{\lambda} :=L^+G t^{\lambda} L^+G/L^+G\]
are called (spherical) affine Schubert cells, and their closures $\overline{\mathrm{Gr}}_G^{\lambda}$ are called (spherical) affine Schubert varieties. We have the following closure relations between the strata: if $\mu, \lambda \in X_*(T)^+$ then \[\mathrm{Gr}_G^{\mu} \subset \overline{\mathrm{Gr}}_G^{\lambda}\] if and only if $\mu \leq \lambda$, that is $\lambda-\mu= \sum a_i \alpha_i^{\vee}$ with $a_i \in \mathbb{Z}^{\geq 0}$.

The above action of $\mathbb{T}$ preserves each affine Schubert cell and each affine Schubert variety, so we have an action of $\mathbb{T}$ on $\overline{\mathrm{Gr}}_G^{\lambda}$. 

The $\mathbb{T}$-fixed subscheme $(\overline{\mathrm{Gr}}_G^{\lambda})^{\mathbb{T}}$ is a reduced discrete scheme: \[(\overline{\mathrm{Gr}}_G^{\lambda})^{\mathbb{T}} = \bigsqcup_{w \in W, \mu \leq \lambda \in X_*(T)^+} t^{w \mu}.\]

In contrast, the $T$-fixed subscheme $(\overline{\mathrm{Gr}}_G^{\lambda})^T$ is a non-reduced scheme such that \[(\overline{\mathrm{Gr}}_G^{\lambda})^T_{\text{red}} = (\overline{\mathrm{Gr}}_G^{\lambda})^{\mathbb{T}}.\] In other words $(\overline{\mathrm{Gr}}_G^{\lambda})^T$ is a non-reduced scheme supported on $ \bigsqcup_{w \in W, \mu \leq \lambda \in X_*(T)^+} t^{w \mu}$.

\subsection{Slices}
As in \cite{KWWY} we define the slices at $\mu$ to $\overline{\mathrm{Gr}}^{\lambda}_G$ as follows: \[\overline{W}_{\mu}^{\lambda} := \overline{\mathrm{Gr}}_G^{\lambda} \cap L^{--}Gt^{\mu}. \] This intersection takes place in the thick affine Grassmannian but we can understand \[\overline{W}_{\mu}^{\lambda} \rightarrow \overline{\mathrm{Gr}}_G^{\lambda}\] as a closed subscheme of the affine Schubert variety. 

\begin{remark} When $\mu$ is not a dominant coweight, our $\overline{W}^{\lambda}_{\mu}$ do not agree with the \textit{generalized slices} found in Coulomb branch literature.
\end{remark}

The slices enjoy the following properties, see \cite[Section 2B,2C]{KWWY}.
\begin{proposition}
\begin{itemize}
\item $\overline{W}_{\mu}^{\lambda}$ are affine schemes of finite type over $\mathbb{C}$, and are normal and Cohen-Macaulay.
\item $\overline{W}_{\mu}^{\lambda}$ carry algebraic actions of $\mathbb{T}$ such that $(\overline{W}_{\mu}^{\lambda})^{\mathbb{T}}=t^{\mu}$, equipped with the reduced scheme structure.
\item $\overline{W}_{\mu}^{\lambda}$ are conical symplectic singularities  satisfying our assumptions from section \ref{symplecticsingassump}. In particular there is a symplectic form on $(\overline{W}_{\mu}^{\lambda})_{reg}$ of $\mathbb{T}$-weight $\varpi_r$. 
\end{itemize}
\end{proposition}

In particular, the $\overline{W}_{\mu}^{\lambda}$ satisfy our assumptions from Section \ref{Sec2} and we have the following corollary.

We fix $\{\underline{x}_{\mu}^{\lambda} \}$ a minimal set of homogeneous generators of the ideal of $I^T(\overline{W}^{\lambda}_{\mu})$; let $n$ be the cardinality of this generating set. We have the following derived intersection 

\begin{equation}
\begin{tikzcd}
\mathbb{V}^{\bullet}(I^T(\overline{W}^{\lambda}_{\mu}), \underline{x}^{\lambda}_{\mu})\arrow[r] \arrow[d] & \overline{W}^{\lambda}_{\mu} \arrow[d] \\
\Spec  \mathbb{C} \arrow["0",r] & (T_{\mu} \overline{W}^{\lambda}_{\mu})^{\neq 0}.\\
\end{tikzcd}
\end{equation}
We will usually shorten the notation by writing simply $\mathbb{V}^{\bullet}(I^T(\overline{W}^{\lambda}_{\mu}))$.
Now, applying Theorem \ref{thm2} and Corollaries \ref{pcare} and \ref{amplitude} to the above diagram, we obtain the following.
\begin{corollary}
The derived intersection $\mathbb{V}^{\bullet}(I^T(\overline{W}^{\lambda}_{\mu})) \rightarrow \overline{W}^{\lambda}_{\mu}$ has structure complex with cohomology concentrated in degrees 

\[[\dim  \overline{W}_{\mu}^{\lambda} - \dim  (T_{\mu} \overline{W}_{\mu}^{\lambda})^{ \neq 0}, 0]. \] In particular, \[(\overline{W}_{\mu}^{\lambda})^{T} \rightarrow \overline{W}_{\mu}^{\lambda} \] is a complete intersection if and only if the dimension $\dim  (T_{\mu} \overline{W}_{\mu}^{\lambda})^{\neq 0}$ is equal to $\dim  (\overline{W}_{\mu}^{\lambda})$. 

Moreover, writing $d= \dim  \overline{W}^{\lambda}_{\mu}$ and $n= \dim  (T_0 \overline{W}^{\lambda}_{\mu})^{\neq 0}$, we have $\mathbb{T}$-equivariant isomorphisms:
\[H^{-i}(\mathscr{O}_{\mathbb{V}^{\bullet}(I^T (\overline{W}^{\lambda}_{\mu}))})^* \simeq H^{d -n+i}(\mathscr{O}_{\mathbb{V}^{\bullet}(I^T(\overline{W}^{\lambda}_{\mu}))})\Bigl(\mathrm{wt}\left(\mathrm{det}((T_0\overline{W}^{\lambda}_{\mu})^{\neq 0})\right)+ \frac{d}{2} \varpi_r\Bigr).\]
\end{corollary}

See section \ref{examples} for many examples.

\section{Derived enhancements and character formula}\label{K-Theory}

\newcommand{\aff}{\mathrm{aff}}
This section is a digression, in which we provide a representation-theoretic meaning for our derived intersections in terms of characters.
Let $G$ be a simply connected simple algebraic group. We fix some basic notions with regard to affine Kac-Moody groups. Let $\widehat{LG}$ be the universal central extension of $LG$, with central torus denoted by $\Gm^{\mathrm{cen}}$. Define the affine Kac-Moody group to be $\widetilde{LG}:=\widehat{LG} \rtimes \Gmrot$; for any subgroup $H$ of $LG$, we also write $\widetilde{H}$ for its preimage in $\widetilde{LG}$. The Kac-Moody maximal torus is $\widetilde{T}=T\times\Gm^{\mathrm{cen}}\times\Gmrot$, and the affine character lattice is $X^*(\widetilde{T})=X^*(T)\oplus \Z \Lambda_0 \oplus \Z \varpi_r$, where $\Lambda_0$ is the generator of $X^*(\Gm^{\mathrm{cen}})$ and $\varpi_r$ is the generator of $X^*(\Gmrot)$. Note that in Kac-Moody literature, $\varpi_r$ is usually denoted $\delta$. Let $W_{\aff}:=W\ltimes X_*(T)$ be the affine Weyl group; we write the translation associated to $\lambda$ in $W_\aff$ by $\tau_\lambda$. For each element $w =w_{\mathrm{f}} \tau_{\lambda} \in W_{\aff}$, we can find some lift $\dot{w}= \dot{w_{\mathrm{f}}}t^{\lambda}$ in $LG$, and we have the Bruhat decomposition
\[\widetilde{LG }= \bigsqcup_{w\in W_\aff} \widetilde{I} \dot{{w}} \widetilde{I},\ \  LG = \bigsqcup_{w \in W_\aff} I \dot{{w}} I.\]
Since we assumed $G$ is simply connected, $W_\aff$ is a Coxeter group generated by simple reflections $\{s_0,\dots,s_n\}$. 

Recall that the Picard group of $\Gr_G$ is isomorphic to $\Z$, and generated by the ample line bundle $\cL_{\Lambda_0}:=\widetilde{LG} \times^{\widetilde{L^+G}}\C_{-\Lambda_0}$; we also have $\cL_{n\Lambda_0}:=\cL_{\Lambda_0}^n$. We will abuse notation and denote the restriction of $\cL_{n\Lambda_0}$ to subschemes of $\Gr_G$ by the same symbol $\cL_{n\Lambda_0}$.

According to \cite[Sec. 8]{Kum}, when $m\ge 0$, the higher cohomology groups $H^i(\overline{\mathrm{Gr}}^{\lambda}, \mathcal{L}_{m \Lambda_0})$ vanish and the dual space to global sections
\[D(m,\lambda):=H^0(\overline{\Gr}_G^{\lambda}, \cL_{m\Lambda_0})^*,\]
is called a (spherical) affine Demazure module. The goal of this section is to give a formula of the $\mathbb{T}$-character of $H^0(\overline{\Gr}_G^{\lambda}, \cL_{m\Lambda_0})$ using the tools in equivariant K-theory, involving our derived intersections $\mathbb{V}^{\bullet}(I^T(\overline{W}^{\lambda}_{\mu}))$.

\subsection{Computations of localization maps}
For $Y$ a (derived) scheme with $H$-action, we write $K^H(Y)$ for the Grothendieck group of $D^{H}_{\coh}(Y)$. For an object $\mathcal{F} \in D^H_{\coh}(Y)$ we write $[\mathcal{F}]$ for the class of $\mathcal{F}$ in $K^H(Y)$. Given an exact functor $F$ between two equivariant derived categories of coherent sheaves, we use the same symbol to denote the induced functor on $K$-groups. Let $\pi: X \rightarrow \mathrm{pt}$ be a proper complex variety with a $\mathbb{T}$-action. Assume the reduced structure $(X^\mathbb{T})_{\text{red}}$ of the $\bT$-fixed subscheme consists of finitely many points. We write $i$ for the closed embedding $(X^\mathbb{T})_{\text{red}} \hookrightarrow X$. In this context, we can use equivariant $K$-theory to obtain character formulas in the following sense: if a $\mathbb{T}$-equivariant coherent sheaf $\mathcal{F}$ on $X$ has cohomology concentrated at degree $0$, then
\begin{equation}\label{eq:character_formula}
	[\Gamma(X,\mathcal{F})]=\pi_*([\mathcal{F}])=\pi_*^\mathbb{T}\circ(i_*)^{-1}([\mathcal{F}])\in K^{\mathbb{T}}(\mathrm{pt})_{\mathrm{loc}},
\end{equation}
which gives a character formula of $\Gamma(X,\mathcal{F})$.

For convenience, we furthermore assume the $\mathbb{T}$-action is linearizable, i.e. for each $\mu \in (X^\mathbb{T})_{\text{red}}$, there exists a $\mathbb{T}$-invariant affine open subset $U_{\mu} \subset X$ containing $\mu$ as the unique $\bT$-fixed point, and 
\[X=\bigcup_{\mu \in (X)^\bT_{\text{red}}} U_{\mu}.\] We write $i_{\mu}$ for the closed embedding $\mu \hookrightarrow U_\mu$, and $j_\mu$ for the open embedding $U_{\mu}\hookrightarrow X$.
Using Thomason localization theorem \cite{Th92}, we have isomorphisms
\[i_*=\bigoplus_{\mu \in (X^\mathbb{T})_{\text{red}}} (j_{\mu}\circ i_{\mu})_*: K^{\mathbb{T}}((X^\mathbb{T})_{\text{red}})_{\mathrm{loc}}=\bigoplus_{\mu \in (X^\mathbb{T})_{\text{red}}} K^{\mathbb{T}}(\mu)_{\mathrm{loc}} \xrightarrow{\simeq} K^{\mathbb{T}}(X)_{\mathrm{loc}},\]
\[i_{\mu*}:K^{\mathbb{T}}(\mu)_{\mathrm{loc}} \xrightarrow{\simeq} K^{\mathbb{T}}(U_{\mu})_{\mathrm{loc}};\]
here $(-)_{\mathrm{loc}}$ means $\otimes_{R(\mathbb{T})} \mathrm{Frac}(R(\mathbb{T}))$. We also have the following commutative diagram, where each arrow is an isomorphism
% https://q.uiver.app/#q=WzAsMyxbMCwwLCJcXGJpZ29wbHVzX3tcXGFscGhhIFxcaW4gKFhee1xcbWF0aGJie1R9fSlfe1xcdGV4dHtyZWR9fX0gS157XFxtYXRoYmJ7VH19KFxcYWxwaGEpX3tcXG1hdGhybXtsb2N9fSJdLFsyLDAsIktee1xcbWF0aGJie1R9fShYKV97XFxtYXRocm17bG9jfX0uIl0sWzEsMSwiXFxiaWdvcGx1c197XFxhbHBoYSBcXGluIChYXlxcbWF0aGJie1R9KV97XFx0ZXh0e3JlZH19fSBLXntcXG1hdGhiYntUfX0oVV97XFxhbHBoYX0pX3tcXG1hdGhybXtsb2N9fSJdLFswLDEsImlfKj1cXGJpZ29wbHVzX3tcXGFscGhhIFxcaW4gKFheXFxtYXRoYmJ7VH0pX3tcXHRleHR7cmVkfX19IChqX3tcXGFscGhhfVxcY2lyYyBpX3tcXGFscGhhfSlfKiJdLFswLDIsIlxcYmlnb3BsdXNfe1xcYWxwaGEgXFxpbiAoWF5cXG1hdGhiYntUfSlfe1xcdGV4dHtyZWR9fX0gaV97XFxhbHBoYSp9IiwxXSxbMSwyLCJcXGJpZ29wbHVzX3tcXGFscGhhIFxcaW4gKFheXFxtYXRoYmJ7VH0pX3tcXHRleHR7cmVkfX19al97XFxhbHBoYX1eKiIsMV1d
\[\begin{tikzcd}
	{\bigoplus_{\mu \in (X^{\mathbb{T}})_{\text{red}}} K^{\mathbb{T}}(\mu)_{\mathrm{loc}}} && {K^{\mathbb{T}}(X)_{\mathrm{loc}}.} \\
	& {\bigoplus_{\mu \in (X^\mathbb{T})_{\text{red}}} K^{\mathbb{T}}(U_{\mu})_{\mathrm{loc}}}
	\arrow["{i_*=\bigoplus_{\mu \in (X^\mathbb{T})_{\text{red}}} (j_{\mu}\circ i_{\mu})_*}", from=1-1, to=1-3]
	\arrow["{\bigoplus_{\mu \in (X^\mathbb{T})_{\text{red}}} i_{\mu*}}"{description}, from=1-1, to=2-2]
	\arrow["{\bigoplus_{\mu \in (X^\mathbb{T})_{\text{red}}}j_{\mu}^*}"{description}, from=1-3, to=2-2]
\end{tikzcd}\]
Thus, the localization map
\[(i_*)^{-1}:  K^{\mathbb{T}}(X)_{\mathrm{loc}} \to \bigoplus_{\mu \in (X^\mathbb{T})_{\text{red}}} K^{\mathbb{T}}(\mu)_{\mathrm{loc}}\]
can be computed by the composition
\[K^{\mathbb{T}}(X)_{\mathrm{loc}} \xrightarrow{\bigoplus_{\mu \in (X^\mathbb{T})_{\text{red}}} j_{\mu}^*} \bigoplus_{\mu \in (X^\mathbb{T})_{\text{red}}} K^{\mathbb{T}}(U_{\mu})_{\mathrm{loc}} \xrightarrow{\bigoplus_{\mu \in (X^\mathbb{T})_{\text{red}}} (i_{\mu *})^{-1}} \bigoplus_{\mu \in (X^\mathbb{T})_{\text{red}}} K^{\mathbb{T}}(\mu)_{\mathrm{loc}}.\]
So it suffices to compute
\[(i_{\mu*})^{-1}:   K^{\mathbb{T}}(U_{\mu})_{\mathrm{loc}} \to K^{\mathbb{T}}(\mu)_{\mathrm{loc}}.\]
%Let $a_{\mu}:U_{\mu}\to \mu$ be the projection to a point. Then we know 
%\[a_{\mu*}\circ i_{\mu*}=\mathrm{id}.\]
Note that
\[\Gamma(U_\mu,i_{\mu*}(-))=\mathrm{id}.
\]
So if we can make sense of $\Gamma(U_\mu,-)$ for elements in $K^{\mathbb{T}}(U_{\mu})_{\mathrm{loc}}$, we find the inverse of $i_{\mu*}$. We have the following tools at our disposal to understand $\Gamma(U_{\mu}, -)$.
\begin{enumerate}
	\item We choose a set of homogeneous elements $x_1,\dots,x_k \in \C[U_{\mu}]$ such that the subscheme cut out by these elements is a finite subscheme set-theoretically supported on $\mu$. Let $K^{\bullet}(x_1,\dots,x_k)$ be the associated Koszul complex. Then for any $[\mathcal{F}]\in K^{\mathbb{T}}(U_{\mu})$, $[\mathcal{F}]\otimes [K^{\bullet}(x_1,\dots,x_k)]$ is supported on the finite thickening of $\mu$ defined by the $x_i$. Thus we can apply $\Gamma(U_\mu,-)$ to this product, and get
	\[(i_{\mu*})^{-1}([\mathcal{F}]\otimes [K^{\bullet}(x_1,\dots,x_k)])=\Gamma(U_\mu,[\mathcal{F}]\otimes [K^{\bullet}(x_1,\dots,x_k)])\in K^{\mathbb{T}}(\mu).\]
	Therefore, we have the following expressions for $(i_{\mu *})^{-1}([\mathcal{F}]):$ 
	\begin{equation}\label{eq:eulerclass}
	\begin{aligned}
		(i_{\mu*})^{-1}([\mathcal{F}])&=\frac{\Gamma(U_\mu,[\mathcal{F}]\otimes [K^{\bullet}(x_1,\dots,x_k)])}{[K^{\bullet}(x_1,\dots,x_k)]}\\
		&=\frac{\Gamma(U_\mu,[\mathcal{F}]\otimes [K^{\bullet}(x_1,\dots,x_k)])}{\prod_{i=1}^k(1-\mathrm{wt}(x_i))}\in K^{\mathbb{T}}(\mu)_{\mathrm{loc}}.
		\end{aligned}
	\end{equation}
	
	\item The second tool relies on the repelling action of $\Gmrot$. Assume $\mathbb{T}$ splits as $T \times \Gmrot$, and the $\Gmrot$-action repels $U_{\mu}$ from $\mu$, i.e. the $\Gmrot$-weights of $\C[U_{\mu}]$ are non-negative. Then for any coherent sheaf $\mathcal{F}$ on $U_{\mu}$, $\Gamma(U_\mu,\mathcal{F})$ is a finitely generated $\C[U_{\mu}]$-module, thus the $\Gmrot$-weights of this module are bounded below. We  write  $\Gamma(U_{\mu},\mathcal{F})_n$ for the subspace of $\Gamma(U_{\mu}, \mathcal{F})$ with $\Gmrot$-weight $n \varpi_r$, and we form the series:
	\begin{equation}\label{eq:a*}
		\Gamma(U_\mu,[\mathcal{F}]):=\mathrm{ch}^T(\Gamma(U_\mu,[\mathcal{F}])):= \bigoplus_{n\in \Z}[\Gamma(U_\mu,\mathcal{F})_n] e^{n\varpi_r} \in K^{T}(\mu)((e^{\varpi_r})).
	\end{equation}
	 Using the finitely generated property of $\Gamma(U_\mu,\mathcal{F})$, the element in \eqref{eq:a*} turns out to lie in $\mathrm{Frac}(R(\mathbb{T}))$, which yields another expression for $(i_{\mu*})^{-1}(\mathcal{F})$.
\end{enumerate}

\subsection{A character formula for $D(m,\lambda)$}
Now we apply the tools in previous subsection to our case. For $\mu$ an arbitrary coweight we write $\mu_+$ for the dominant translate and $w_{\mu}$ for the minimal representative in $W$ effecting this translation: $w_{\mu}\mu_+=  \mu$.

We have $$(\overline{\Gr}_G^{\lambda})^{\mathbb{T}}=\bigsqcup_{\mu_+ \leq \lambda} t^{\mu}.$$
 %we will denote this set of cocharacters also by $\mathcal{G}^{\lambda}$ {\color{red}(is there some better notation?)}. 
For each such $\mu$, there is an open neighborhood of $t^{\mu}$ ($\bT$-equivariantly) of the form 
%isomorphic to $\Gr^{\mu}_G \times \overline{\cW}^{\lambda}_{\mu}$ (see \cite[Prop. 2.3.9]{ZhuGrass}), and all such open subsets form an open cover of $\overline{\Gr}_{G}^{\lambda}$. For the dominant coweight $\mu$, define an affine open neighborhood of $t^{\mu}$ to be
$$U^{\lambda}_{\mu}:=w_{\mu}(I) \cdot t^{\mu} \times \overline{\cW}^{\lambda}_{\mu}.$$ The collection of $U_{\mu}^{\lambda}$ form an open covering of $\overline{\mathrm{Gr}}^{\lambda}$.
 Observe that $w_{\mu} (I)\cdot t^{\mu}$ is an affine space isomorphic to the tangent space $T_{{\mu}} \Gr^{\mu_+}$. 
Note that 
$$T_{\mu}\Gr^{\mu_+}=\mathfrak{g}[[t]]/(\mathrm{Ad}_{t^\mu}\mathfrak{g}[[t]]\cap \mathfrak{g}[[t]]);$$
the set of $\mathbb{T}$-weights of this vector space is
\[\{\alpha + n \varpi_r \mid \alpha \in \Phi_G, n\in \mathbb{N},\text{ such that } 0\le n\le \langle\check{\alpha},\mu\rangle-1\}.\]

Working $\mathbb{T} \times \mathbb{G}_m^{\mathrm{cen}}$ equivariantly, and using the formula \cite[Chapter 6]{Kac}, the restriction of the line bundle $\cL_{m\Lambda_0}$ to $U_{\mu}^{\lambda}$ is $$\mathscr{O}_{U^{\lambda}_{\mu}}\left(-\tau_{\lambda}\cdot m\Lambda_0\right)=\mathscr{O}_{U^{\lambda}_{\mu}}\left(-m(\Lambda_0+\iota(\mu)-\tfrac12 (\mu,\mu) \varpi_r)\right).$$
Here $(,)$ is the $W$-invariant bilinear form on $X_*(T)_{\mathbb{R}}:=X_*(T)\otimes_{\mathbb{Z}}\mathbb{R}$ such that the norm of each long root is $2$, and $\iota:X_*(T)_{\mathbb{R}} \to X^*(T)_{\mathbb{R}}$ is the map determined by $(,)$. As we work only $\mathbb{T}$-equivariantly, we may write 
\[\mathscr{O}_{U^{\lambda}_{\mu}}\left(-m(\iota(\mu)-\tfrac12 (\mu,\mu) \varpi_r)\right).\]

Combining \eqref{eq:eulerclass} and \eqref{eq:a*}, we obtain the localization formula for $[\mathscr{O}_{U^{\lambda}_{\mu}}]\in K^{\mathbb{T}}(U^{\lambda}_{\mu})$:
\begin{align*}
(i_{\mu*})^{-1}([\mathscr{O}_{U^{\lambda}_{\mu}}])&=(i_{\mu*})^{-1}([\mathscr{O}_{T_{t^\mu}\Gr^{\mu_+}}\boxtimes \mathscr{O}_{\overline{W}_{\mu}^{\lambda}}])\\
&=\frac{\mathrm{ch}^T(\C[\overline{W}_{\mu}^{\lambda}])}{\prod_{\alpha \in \Phi_G, 0\le n\le \langle\alpha,\mu\rangle-1}(1-e^{-n \varpi_r} e^{-\alpha})}\in \mathrm{Frac}(R(T)((e^{\varpi_r}))),
\end{align*}
which turns out to lie in $K^{\mathbb{T}}(t^{\mu})_{\mathrm{loc}}=\mathrm{Frac}(R(\mathbb{T}))$; we will just write the denominator as $\mathrm{eu}(T^*_{\mu}\Gr^{\mu_+})$, i.e. the Euler class of the $\mathbb{T}$-vector space $T^*_{\mu}\Gr^{\mu_+}$. Then the localization formula for $[\mathcal{L}_{m \Lambda_0}] \in K^{\mathbb{T}}(\overline{\Gr}_{G}^{\lambda})$ is 
\begin{equation*}\label{}
	\begin{aligned}
		(i_*)^{-1}(\mathcal{L}_{m \Lambda_0})&=\sum_{\mu\in (\overline{\Gr}_G^\lambda)^\bT} (i_{\mu *})^{-1} (\mathscr{O}_{U^{\lambda}_{\mu}}(-m(\iota(\mu) -\tfrac12(\mu,\mu)\varpi_r)))\\
		&=\sum_{\mu\in (\overline{\Gr}_G^\lambda)^\bT}e^{\frac{1}{2}m(\mu,\mu)\varpi_r}e^{-m\iota(\mu)}\frac{\mathrm{ch}^T(\C[\overline{W}_{\mu}^{\lambda}])}{\mathrm{eu}(T^*_{\mu}\Gr^{\mu_+})}t^{\mu}\in K^{\mathbb{T}}\left(\bigsqcup_{\mu \in (\overline{\Gr}_G^\lambda)^\bT} t^{\mu}\right)_{\mathrm{loc}}.
	\end{aligned}
\end{equation*}
We finally conclude using \eqref{eq:character_formula}: 

\begin{equation*}\label{eq:formula}
	[\Gamma(\overline{\Gr}_G^{\lambda},\cL_{m\Lambda_0})]=\sum_{\mu\in (\overline{\Gr}_G^\lambda)^\bT}e^{\frac{1}{2}m(\mu,\mu)\varpi_r}e^{-m\iota(\mu)}\frac{\mathrm{ch}^T(\C[\overline{W}_{\mu}^{\lambda}])}{\mathrm{eu}(T^*_{\mu}\Gr^{\mu_+})}.
\end{equation*}

Furthermore, we can use our derived enhancement $\mathbb{V}^\bullet (I^T(\overline{W}_\mu^\lambda))$ to describe $\mathrm{ch}^T(\C[\overline{\cW}_{\mu}^{\lambda}])$. As before, we choose a set of homogeneous generators $\{\underline{x}_\mu^\lambda\}$. Then according to \eqref{eq:eulerclass}, 
\[\mathrm{ch}^T(\C[\overline{\cW}_{\mu}^{\lambda}])=\frac{\mathrm{ch}^T(K^\bullet_{\C[\overline{\cW}_\mu^\lambda]}(\underline{x}_\mu^\lambda))}{\prod(1-\mathrm{wt}(x_\mu^\lambda))}=\frac{\mathrm{ch}^T(\C[\mathbb{V}^\bullet (I^T(\overline{W}_\mu^\lambda))])}{\prod(1-\mathrm{wt}(x_\mu^\lambda))}.\]
By Proposition \ref{prop:generator}, the images of $\{\underline{x}_\lambda^\mu\}$ in $(T^*_{\mu} \overline{\cW}_\mu^\lambda)^{\ne 0}$ form a basis, thus $\prod(1-\mathrm{wt}(x_\mu^\lambda))$ is exactly the Euler class $\mathrm{eu}((T^*_{\mu} \overline{\cW}_\mu^\lambda)^{\ne 0})$. We finally get the following character formula.

\begin{proposition} We have the following equality in $K^{\mathbb{T}}(\mathrm{pt})$:
\[[\Gamma(\overline{\Gr}_G^{\lambda},\cL_{m\Lambda_0})]=\sum_{\mu\in (\overline{\Gr}_G^\lambda)^\bT}e^{\frac{1}{2}m(\mu,\mu)\varpi_r}e^{-m\iota(\mu)}\frac{\mathrm{ch}^T(\C[\mathbb{V}^\bullet (I^T(\overline{W}_\mu^\lambda))])}{\mathrm{eu}(T^*_{\mu}\Gr^{\mu_+})\cdot \mathrm{eu}((T^*_{\mu} \overline{\cW}_\mu^\lambda)^{\ne 0})}.\]
\end{proposition}

\subsection{A character formula from resolutions}

We may contrast the character formula above with a character formula coming from Bott-Samelson-Demazure-Hansen resolutions. Set $P_{i}:=I \dot{s_i} I \sqcup I\subset LG$ to be the parahoric corresponding to a simple reflection $s_i$ for $i=0,1,\dots,n$.

Let $\mathfrak{w}=s_{i_1} s_{i_2} \dots s_{i_k}$ be a reduced expression for $\tau_\lambda \in W_{\aff}$. Define
\[Z_{\mathfrak{w}}:=P_{i_1}\times^I P_{i_2}\times^I \dots P_{i_k}/I, \]
in other words, $Z_{\mathfrak{w}}$ is defined to be the quotient of $P_{i_1}\times\dots \times P_{i_k}$ by the right $I^{\times k}$-action:
\[(p_1,p_2,\dots,p_k)\cdot (b_1,b_2,\dots,b_k)=(p_1 b_1,b_1^{-1} p_2 b_2\dots,b_{n-1}^{-1}p_nb_n).\]
Then we have the Bott-Samelson-Demazure-Hansen (BDSH) resolution
\[\pi_{\mathfrak{w}}:Z_{\mathfrak{w}} \rightarrow \overline{\Gr}^{\lambda},\] \[[p_1,\dots,p_k]\mapsto [p_1\dots p_k];\]
c.f. \cite[Def. 7.1.3, Thm. 8.2.2]{Kum}.

We have a $\mathbb{T}$-action on $Z_{\mathfrak{w}}$ by multiplication on the left and the resolution $\pi_{\mathfrak{w}}$ is $\mathbb{T}$-equivariant. According to \cite[Thm. 8.2.2]{Kum}, $\pi_{\mathfrak{w}*} \mathscr{O}_{Z_{\mathfrak{w}}} \simeq \mathscr{O}_{\overline{Gr}^{\lambda}}$, thus by the projection formula we have
\[R \Gamma(\overline{\Gr}^\lambda, \cL_{m\Lambda_0}) \simeq R\Gamma(Z_\mathfrak{w},\pi_{\mathfrak{w}}^*\cL_{m\Lambda_0}).\] 

Since $Z_{\mathfrak{w}}$ is smooth with isolated fixed points, we may apply localization to compute this class. The following lemma is not hard to prove. %The $\mathbb{T}$-fixed points of $Z_{\mathfrak{w}}$ are indexed by subwords of $\mathfrak{w}$. 
\begin{lemma}
	The $\bT$-fixed points of $Z_{\mathfrak{w}}$ corresponds to the subwords of $\mathfrak{w}$: 
	\[Z_\mathfrak{w}^\bT=\{[s_{i_1}^{e_1},\dots,s_{i_k}^{e_k}] \colon e_j\in \{0,1\},j=1,\dots,k\}=\{\text{subwords } \mathfrak{v} \text{ of }\mathfrak{w}\},\]
	The $\bT$-weights of cotangent space $T^*_{\mathfrak{v}}Z_{\mathfrak{w}}$ for $\mathfrak{v}=[s_{i_1}^{e_1},\dots,s_{i_k}^{e_k}]\in Z_\mathfrak{w}^\bT$ are 
	\begin{equation*}
		\{s_{i_1}^{e_1}\dots s_{i_j}^{e_j} (\alpha_{i_j}),j=1,\dots,k\}\footnotemark;
	\end{equation*}
	\footnotetext{This is a character of $\widetilde{T}$; we restrict it to a $\mathbb{T}$-character; we will similarly abuse notations later.}
	here, $\alpha_i$ is the simple affine root corresponding to the simple reflection $s_i$. The $\bT$-weight of the fiber $\pi_\mathfrak{w}^*\cL_{m\Lambda_0}|_{\mathfrak{v}}$ is $-mv(\Lambda_0)$ for $v:=s_{i_1}^{e_1}\dots s_{i_k}^{e_k}$, the image in $W_{\aff}$ of the word $\mathfrak{v}$.
\end{lemma}
Again by localization (now on $Z_{\mathfrak{w}}$) we obtain a character formula 
\begin{equation*}
	\begin{aligned}
		[\Gamma(\overline{\Gr}^\lambda, \cL_{m\Lambda_0})]&=[\Gamma(Z_\mathfrak{w},\pi_{\mathfrak{w}}^*\cL_{m\Lambda_0})] \\ &= \sum_{\mathfrak{v}} \frac{e^{-mv(\Lambda_0)}}{\mathrm{eu} (T^*_{\mathfrak{v}} Z_{\mathfrak{w}})}\\
		&= \sum_{\mathfrak{v}}\frac{e^{-mv(\Lambda_0)}}{\mathrm{eu} (T^*_{\mathfrak{v}} Z_{\mathfrak{w}})}.
	\end{aligned}
\end{equation*}
We can rearrange the summation as
\[[\Gamma(\overline{\Gr}^\lambda, \cL_{m\Lambda_0})]=\sum_{\mu\in (\overline{\Gr}_G^\lambda)^\bT}\sum_{\mathfrak{v}\in \pi_{\mathfrak{w}}^{-1}(\mu)}\frac{e^{\frac{1}{2}m(\mu,\mu)\varpi_r}e^{-m\iota(\mu)}}{\mathrm{eu} (T^*_{\mathfrak{v}} Z_{\mathfrak{w}})}.\]
We find it interesting to compare this character formula with the previous character formula \eqref{eq:eulerclass}. The individual terms in the BSDH character formula are simple and familiar, involving weights of a line bundle at smooth fixed points in $Z_{\mathfrak{w}}$ and their cotangent spaces. The price to pay is the difficult Weyl group combinatorics of describing the words $\mathfrak{v} $ in the fiber $ \pi^{-1}_{\mathfrak{w}}(\mu)$.  In contrast, the character formula \eqref{eq:eulerclass} has terms indexed by the $T$-fixed points of $\overline{\mathrm{Gr}}^{\lambda}$, but involve the euler characteristic of our derived intersections $\mathbb{V}^{\bullet}(I^T(\overline{W}^{\lambda}_{\mu}))$. It seems the derived intersection $\mathbb{V}^{\bullet}(I^T(\overline{W}^{\lambda}_{\mu}))$ is packaging together the contributions of the various $\mathfrak{v} \in \pi^{-1}_{\mathfrak{w}}(\mu)$.

\section{Generalized Minors}\label{genminors}
There is one family of slices where we can be particularly explicit. This is the case where $G$ is $SL_{n+1}$ and $\mu=0$. In this case, the slice $\overline{W}^{\lambda}_{0}$ is an open dense affine neighborhood of $t^0$ in $\overline{\mathrm{Gr}}_{SL_{n+1}}^{\lambda}$. We begin with some material on matrix coefficients and then specialize to the $G=SL_{n+1}$ situation.

Let $G$ be a reductive algebraic group over $\mathbb{C}$. Let $V$ be an irreducible finite-dimensional algebraic representation. We have the following matrix coefficient functions $\Delta_{v_1,v_2} \in \mathscr{O}_G(G)$ for $v_1 \in V^*$, $v_2 \in V$: \[\Delta_{v_1, v_2}(g):=\langle v_1, g v_2 \rangle.\]  

By change of base we can also understand \[\Delta_{v_1, v_2} : G(R((t^{-1}))) \rightarrow R((t^{-1})),\] and we write $\Delta_{v_1,v_2}^{(s)}$ for the coefficients: for $g \in G(R((t^{-1})))$ we have
\[\Delta_{v_1, v_2}(g)=\sum_{-\infty}^N \Delta_{v_1, v_2}^{(s)}(g)t^{-s}.\]

 Let $J_0^{\lambda}$ be the ideal in $\mathscr{O}(L^{--}G)$ generated by $\Delta_{\beta, \nu}^{(s)}$ as $\beta, \nu$ range over weight bases for $V_{\varpi_i}^*=V_{\varpi_i^*}$ and $V_{\varpi_i}$ and where $s > \langle \lambda, \varpi_{i}^* \rangle$. Recall that we can embed $\overline{W}_0^{\lambda}$ as a closed subscheme of $L^{--}G$. When $G=SL_{n+1}$,  $\mathbb{C}[\overline{W}^{\lambda}_0]$ has a straightforward description in terms of these generalized minor matrix coefficients. Note that such an easy description is known not to hold in other types, for instance $E_6, E_7$ and $E_8$, \cite{BHY}.
Thus from now on, we will fix $G= SL_{n+1}$.
\begin{theorem} Let $G=SL_{n+1}$.
The ideal of $\overline{W}_0^{\lambda}$ in $\mathscr{O}(L^{--}G)$ is $J_0^{\lambda}$.
\end{theorem}

\begin{proof}
This was conjectured in \cite{KWWY} and is the main result of \cite{KMWY}.
\end{proof}
We will make a minor change of notation (which also appears in \cite{KMWY}). Since the rest of the paper will only concern itself with affine Schubert varieties for $SL_{n+1}$ we switch from fundamental weight notation to more direct generalized minor notation. 

\begin{definition}
Fix $n$ a positive integer. Let $I$,$J$ be two subsets of $k$ distinct elements $I,J \subset \{1,2, \dots n\}$. We write \[\Delta_{I,J}^{(s)}: L^{--}G \rightarrow \mathbb{A}^1\] which sends an $R$-point $M \in L^{--}G(R)$ to the coefficient of $t^{-s}$ in $det_{I,J}(M) \in R[[t^{-1}]]$. 

\end{definition}

\begin{proposition}
The ideal $J_0^{\lambda}$ in $\mathscr{O}(L^{--}G)$ is generated by $\Delta_{I,J}^{(s)}$, where $1 \leq |I|=|J| \leq n$ and $s > \langle \lambda, \omega_{n+1-|I|} \rangle$.
\end{proposition}
\begin{proof}
This is just the theorem above restated in different notation.
\end{proof}

\begin{example}
Let us consider $\lambda=\alpha_1^{\vee}+2\alpha_2^{\vee}+2\alpha_3^{\vee}$. Then $\overline{W}_0^{\lambda}(R)$ is the set of matrices

\[M=\begin{pmatrix} 1+ a_1t^{-1}+a_2t^{-2} & b_1t^{-1}+b_2t^{-2} & c_1t^{-1}+c_2t^{-2}& d_1t^{-1} +d_2t^{-2} \\ h_1t^{-1}+h_2t^{-2} & 1+ a'_1t^{-1}+a'_2t^{-2} & e_1t^{-1} +e_2t^{-2} & f_1t^{-1}+f_2t^{-2} \\ i_1t^{-1}+i_2t^{-2} & j_1t^{-1}+j_2t^{-2} & 1+a^{''}_1t^{-1}+a^{''}_2t^{-2} & g_1t^{-1}+g_2t^{-2} \\ k_1t^{-1}+k_2t^{-2} & l_1t^{-1}+l_2t^{-2} & m_1t^{-1}+m_2t^{-2} & 1+a^{'''}_1t^{-1}+a_2^{'''}t^{-2} \\ \end{pmatrix}\] defined by the conditions that $a_i, b_i$ etc. belong to $R$, the valuation (with respect to $t$) of all 2x2 minors be greater than or equal to -2, the valuation of all 3x3 minors be greater than or equal to -1, and the determinant be equal to 1. The coefficient $2$ of $\alpha_3^{\vee}$ is used to bound the valuations of the 1x1 minors.
\end{example}

In this case, it is easy to describe $(T_0 \overline{W}_0^{\lambda})^{\neq 0}$ and $(T_0^* \overline{W}_0^{\lambda})^{\neq 0}$.
\begin{proposition}\label{prop:min ai}
Let $G=SL_{n+1}$, and let $\lambda = \sum a_i \alpha_i^{\vee}$ be a dominant coweight. The vector space $(T_{0}^* \overline{W}_0^{\lambda} )^{\neq 0}= (\mathfrak{m}_A/\mathfrak{m}_A^2)^{\neq 0}$ has a basis given by the images (all of which are nonzero) of the matrix coefficients  \[\{\Delta_{i \neq j}^{(l)} \}_{1 \leq l \leq \mathrm{min\{}a_i\}}\] for $1 \leq l \leq \mathrm{min}(a_i)$. Note moreover that this minimum is either $a_1$ or $a_n$.

By Nakayama's lemma in the graded case, the ideal $I^T(\overline{W}_0^{\lambda})$ is minimally generated by these matrix coefficients: 

\[I^T(\overline{W}_0^{\lambda}) = ( \{\Delta_{i \neq j}^{(l)}\}_{1 \leq l \leq \mathrm{min} \{a_i\}} )\subset A=\mathbb{C}[\overline{W}_0^{\lambda}].\]
\end{proposition}
\begin{proof}
For a coroot $\alpha \in \Phi^{\vee}$, assume $t^{k\alpha} \in \overline{\mathrm{Gr}}^{\lambda}$. Then $t^{k \alpha}$ is in the closure of the locally closed subscheme \footnote{Recall that we write $\varpi_r$ for the generator of $X^*(\Gmrot)$; this is usually written $\delta$ in the Kac-Moody literature.}
\[\mathbb{A}^1 \rightarrow \overline{W}_0^{\lambda}\]
\[x \mapsto U_{\alpha-k\varpi_r}(x). \]
Such curves in $\overline{W}_0^{\lambda}$ lead to a supply of tangent vectors $X_{\alpha-k\varpi_r}$ called root tangent vectors. We have the subspace spanned by root tangent vectors 
\[(T_0\overline{W}_0^{\lambda})^{\text{root}} \subset (T_0 \overline{W}_0^{\lambda})^{ \neq 0}.\] Dually, taking the images $\overline{\Delta}_{i, j}^{(k)} \in \mathfrak{m}_A/\mathfrak{m}_A^2$ of matrix coefficients $\Delta_{i, j}^{(k)}$ yields a supply of cotangent vectors dual to the root tangent vectors. More precisely we have the following diagrams: 

\begin{center}
\begin{tikzcd}
\mathbb{A}^1 \arrow["{U_{\alpha-k \varpi_r}}",r] & \overline{W}_0^{\lambda} \arrow["{\Delta_{i,j}^{(k)}}",d] \\
 & \mathbb{A}^1 \\
\end{tikzcd}
\end{center}

such that the composition 
\[\Delta_{i,j}^{(k)} \circ U_{\alpha-k\varpi_r}: \mathbb{A}^1 \rightarrow \mathbb{A}^1\] 
is either the identity (if $\alpha=\epsilon_i-\epsilon_j, k=l$) or $0$ otherwise.
For general types the containment $(T_0\overline{W}_0^{\lambda})^{\text{root}} \subset (T_0 \overline{W}_0^{\lambda})^{\neq 0}$ may be strict but for $G=SL_{n+1}$, by work of \cite{KMWY} and \cite{KPZ}, we have an equality 
\[(T_0 \overline{W}^{\lambda}_0)^{\text{root}} = (T_0 \overline{W}_0^{\lambda})^{\neq 0}.\] Thus for $G=SL_{n+1}$, $A=\mathbb{C}[\overline{W}_0^{\lambda}]$ we observe that
\[(T^*_0 \overline{W}_0^{\lambda})^{\neq 0} =(\mathfrak{m}_A/\mathfrak{m}_A^2)^{\neq 0}\] 
has a homogeneous basis given by $\{\overline{\Delta}_{a,b}^{(k)}\}_{a \neq b}$ where $k$ ranges between $1$ and the maximal number such that $t^{k \alpha} \in \overline{\mathrm{Gr}}^{\lambda}$. Note that by the $G$-symmetry on $\overline{\mathrm{Gr}}^{\lambda}$,  $t^{k \alpha} \in \overline{\mathrm{Gr}}^{\lambda} \iff t^{k \beta} \in \overline{\mathrm{Gr}}^{\lambda}$ for any $\alpha, \beta \in \Phi^{\vee}$.

With this in mind,  it is enough to study when $t^{k \theta^{\vee}} \in \overline{\mathrm{Gr}}^{\lambda}$. Recall that in type $A$, 
\[\theta^{\vee}= \sum_i \alpha_i^{\vee}.\] 
In particular, we have
\[k \theta^{\vee} \leq \lambda \iff k \leq \mathrm{min}_i \{a_i\}.\]
 Moreover, by Lemma \ref{dominantcharacterization} below, this minimum is always attained by either $a_1$ or $a_n$.
We conclude that for $G=SL_{n+1}$, the images $\overline{\Delta}_{i \neq j}^{(k)}$ for $k \leq \mathrm{min}_i\{a_i\}= \mathrm{min}\{a_1, a_n\}$ span $(\mathfrak{m}_A/\mathfrak{m}_A^2)^{ \neq 0}$.

\end{proof}

\begin{lemma}\label{dominantcharacterization} Let $G= SL_{n+1}$. 
Let $\lambda=\sum a_i \alpha_i^{\vee}$ be a dominant coweight. Then the $a_i$ form a convex sequence. In particular, the minimal coefficient is either $a_1$ or $a_n$. 
\end{lemma}

\begin{proof}
Left to the reader (who can glance at section \ref{globaltheorems} for a hint).
\end{proof}

\section{Examples}\label{examples}

We provide several examples in low rank, computed with the help of Macaulay2 \cite{M2}.
\subsection{$G=SL_2$}

\begin{example}
For $\lambda= n \alpha_1^{\vee}$, $n \in \mathbb{Z}^{\geq 0}$, $\mathbb{V}^{\bullet}(I^T(\overline{W}_0^{\lambda}))$ is classical; that is, the derived scheme $\mathbb{V}^{\bullet}(I^T(\overline{W}_0^{\lambda}))$ is isomorphic to $(\overline{W}^{\lambda}_0)^T.$ %$(\overline{\mathrm{Gr}}_{SL_2}^{\lambda})_0^{T, \bullet}$ is 0-coconnective.
\end{example}

We have explicit coordinates for $\overline{W}_0^{\lambda}$; an $R$-point is given by \[M=\begin{pmatrix} 1+a_1t^{-1}+a_2 t^{-2} \dots +a_n t^{-n} & b_1t^{-1} + \dots + b_nt^{-n} \\ c_1t^{-1} + \dots + c_nt^{-n} & 1+d_1t^{-1} + \dots d_nt^{-n} \\ \end{pmatrix}\]
where $a_i, b_i, c_i, d_i\dots \in R$ and such that $\mathrm{det}(M)=1$. It is not difficult to see that $M \in (\overline{W}_0^{\lambda})^T$ if and only if the $b_i$ and $c_i$ all vanish; in other words $M \in (\overline{W}_0^{\lambda})^T(R)$ if and only if $\Delta_{1,2}^{(i)}(M)=0$ and $\Delta_{2,1}^{(i)}(M)=0$ for all $1 \leq i \leq n$. These $2n$ elements minimally generate the ideal $I^T(\overline{W}_0^{\lambda})$. Since $\mathbb{C}[\overline{W}_0^{\lambda}]/I^T(\overline{W}_0^{\lambda})$ is 0-dimensional and since $\dim  \overline{W}_0^{\lambda}=2n$, the collection $\{\Delta_{1,2}^{(i)}, \Delta_{2,1}^{(i)}\}$ is a homogeneous system of parameters. Since $\overline{W}_0^{\lambda}$ is Cohen-Macaulay, these elements form a regular sequence. The derived scheme defined by $K^{\bullet}(\Delta^{(1)}_{1,2}, \Delta_{2,1}^{(1)}, \dots, \Delta^{(n)}_{1,2}, \Delta^{(n)}_{2,1})$  has only homology in degree 0, and $\mathbb{V}^{\bullet}(I^T(\overline{W}^{\lambda}_0))$ is classical.

Since the derived scheme $\mathbb{V}^{\bullet}(I^T(\overline{W}^{\lambda}_0))$ is always classical, Corollary \ref{sympoly} implies that the Poincar\'e polynomial is always a palindromic polynomial. In this case it coincides with the $\Lambda_0$-string of the level one affine Demazure module $D(1, \lambda)$. In this case, a stronger result is known; the $\Lambda_0$ strings of level 1 affine Demazure characters for $SL_2$ are q-binomial coefficients, which are always palindromic.

\subsection{$G=SL_3$}
\begin{example}\label{sl3example}
Consider $\theta^{\vee}=\alpha_1^{\vee}+\alpha_2^{\vee}$ for $SL_3$. The derived scheme $\mathbb{V}^{\bullet}(I^T(\overline{W}^{\theta^{\vee}}_0))$ is not classical.
\end{example}

For $R$ a $\mathbb{C}$-algebra, an $R$-point of $\overline{W}_{0}^{\theta^{\vee}}$ is given by \[\begin{pmatrix}1 + a_1t^{-1} & b_1 t^{-1} & c_1 t^{-1} \\ d_1 t^{-1} & 1+ e_1 t^{-1} & f_1 t^{-1} \\ g_1 t^{-1} & h_1 t^{-1} & 1+k_1t^{-1} \\ \end{pmatrix}\] such that $\mathrm{det} M = 1$ and the valuation of all $2x2$ minors is greater than or equal to $-1$. It is easy to see that $M \in (\overline{W}_{0}^{\theta^{\vee}})^T$ if and only if $b_1= c_1= f_1= d_1=g_1=f_1=0$, or in other words $\Delta_{1,2}^{(1)}(M)=\Delta_{1,3}^{(1)}(M)= \dots =0$. Less obvious is that $\Delta_{1,2}^{(1)}, \Delta_{1,3}^{(1)}$, $\Delta_{2,3}^{(1)}$, $\Delta_{2,1}^{(1)}$, $\Delta_{3,1}^{(1)}$ and $\Delta_{3,2}^{(1)}$ minimally generate $I^T(\overline{W}_{0}^{\theta^{\vee}})$, but this is true by Proposition \ref{prop:min ai} and Proposition \ref{prop:generator}. Since $\overline{W}_{0}^{\theta^{\vee}}$ is $4$-dimensional, the $T$-fixed subscheme is not a complete intersection. The derived intersection $\mathbb{V}^{\bullet}(I^T(\overline{W}^{\theta^{\vee}}_0))$ associated to the CDGA $K_A^{\bullet}(\Delta_{1,2}^{(1)}, \Delta_{1,3}^{(1)}, \Delta_{2,3}^{(1)}, \Delta_{2,1}^{(1)}, \Delta_{3,1}^{(1)},\Delta_{3,2}^{(1)})$ has cohomology concentrated in degrees $[-2,0]$. We give the Poincar\'e polynomial in graded form. The horizontal grading is the homological grading and the vertical grading is the $\Gmrot$-weight.

\begin{center}
\begin{tabular}{ c c c | c}
1 &  &  & (4) \\
2 & & & (3) \\
& 6 & & (2) \\
 &  & 2 & (1) \\
 & & 1 & (0) \\
\hline
$H^{-2}$ & $H^{-1}$ & $H^{0}$ & \\
\end{tabular}
\end{center}

\begin{remark}\label{sl3remark}
Note that this derived intersection is not a $T$-fixed derived scheme. The $6$ generators in degree $(-1, 2)$ have nontrivial $T$-weight, given by weights of off-diagonal $2x2$ generalized minors.
\end{remark}

\begin{example} The derived intersection $\mathbb{V}^{\bullet}(I^T(\overline{W}_0^{\alpha_1^{\vee}+2\alpha_2^{\vee}}))$ is classical (in other words, isomorphic to $(\overline{W}_0^{\alpha_1^{\vee}+2\alpha_2^{\vee}})^T)$.
\end{example}

For a $\mathbb{C}$-algebra $R$, an $R$ point of $M \in \overline{W}_{0}^{\alpha_1^{\vee} +2 \alpha_2^{\vee}}(R)$ is given by \[\begin{pmatrix} 1+ a_1t^{-1}+a_2t^{-2} & b_1t^{-1}+b_2t^{-2} & c_1t^{-1}+c_2t^{-2} \\ d_1t^{-1}+d_2t^{-2} & 1+ e_1t^{-1}+e_2t^{-2} & f_1t^{-1}+f_2t^{-2} \\ g_1t^{-1}+g_2t^{-2} & h_1t^{-1}+h_2t^{-2} & 1+ k_1t^{-1}+k_2t^{-2} \\\end{pmatrix}\] such that all $2 \times 2$ minors have valuation of determinant $\geq -1$, and such that the determinant of $M$ is 1. The $R$-point $M$ factors through the $T$-fixed locus $M \in (\overline{W}_{0}^{\alpha_1^{\vee}+2 \alpha_2^{\vee}})^T_{0}$ if and only if all off-diagonal terms vanish, e.g. if and only if $\Delta_{i,j}^{(l)}(M)=0$ for $i \neq j, 1 \leq l \leq 2$. In this case however, the ideal $I^T(\overline{W}_{0}^{\alpha_1^{\vee}+2 \alpha_2^{\vee}})$ is minimally generated by $\{ \Delta_{i,j}^{(1)}\}$ for $i \neq j$. For instance, the valuation condition on the upper right $2x2$ minor implies that the coefficient of $t^{-2}$ vanishes; expanding in coordinate functions we have the relation $\Delta_{1,2}^{(1)}\Delta_{2,3}^{(1)}-\Delta_{1,3}^{(2)}-\Delta_{2,2}^{(1)} \Delta_{1,3}^{(1)}=0$ in $\mathbb{C}[\overline{W}_{0}^{\alpha_1^{\vee}+2\alpha_2^{\vee}}]$, so we see that $\Delta_{1,3}^{(2)}$ is in the ideal generated by the $\Delta_{i,j}^{(1)}$, $i \neq j$. Since the ideal of the $T$-fixed subscheme is minimally generated by $6$ elements $\Delta_{i,j}^{(1)}$ and since $\dim  \overline{W}_{0}^{\alpha_1^{\vee}+2\alpha_2^{\vee}}=6$, and since the $T$-fixed subscheme 0 dimensional, the six generators of the ideal form a homogeneous system of parameters. Since $\overline{W}^{\lambda}_{\mu}$ are Cohen-Macaulay, this homogeneous system of parameters is a regular sequence, and the derived intersection is classical. The graded Poincar\'e polynomial as as follows:

\begin{center}
\begin{tabular}{  c | c}
1& (3) \\
 2& (2) \\
 2& (1) \\
 1& (0) \\
\hline
 $H^{0}$ & \\
\end{tabular}
\end{center}

These are the Betti numbers of $H^*(SL_3/B)$, as well as the string below $\Lambda_0$ of the affine Demazure module $D(1, \alpha_1^{\vee}+2 \alpha_2^{\vee})$.

\begin{example}Derived structure on $\mathbb{V}^{\bullet}(I^T(\overline{W}^{2 \theta^{\vee}}_0))$. \label{2theta}
\end{example}
The dimension of $\overline{W}_{0}^{2 \theta^{\vee}}$ is 8, and the minimal number of generators of $I^T(\overline{W}^{2 \theta^{\vee}}_0)$ is $6 \cdot 2=12$. Thus our Corollary \ref{amplitude} above states that the derived intersection $\mathbb{V}^{\bullet}(I^T(\overline{W}^{2 \theta^{\vee}}_0))$ has structure sheaf with cohomologies concentrated in degrees $[-4,0]$. The Poincar\'e polynomial is as follows:

\begin{center}
\begin{tabular}{ c c c c c | c} 
1 & & &  & &  (14) \\
2 & & &  & &  (13) \\
5 & & &  & & (12) \\
4 & 6 & & & &  (11)\\
3 & 18 & &  & & (10)\\
 & 24 & 2 & & & (9) \\
 &12 &21  & & & (8) \\
 & &44 & & & (7) \\
 &  &21 &12 &  & (6) \\
 & &2 &24 & & (5)\\
 & & &18 &3 & (4) \\
 & & &6 &4 & (3)\\
 & & & & 5& (2) \\
 & & & & 2& (1) \\
 & & & & 1& (0)\\
 \hline
 $H^{-4}$ & $H^{-3}$ & $H^{-2}$ & $H^{-1}$ & $H^0$ & \\
\end{tabular}
\end{center}

\subsection{$G=SL_4$}
\begin{example} \label{sl4}
The derived intersection $\mathbb{V}^{\bullet}(I^T(\overline{W}^{\theta^{\vee}}_0))$.
\end{example}
The ideal of the $T$-fixed subscheme in $\overline{W}_{0}^{\theta^{\vee}}$ is minimally generated by $\Delta_{i,j}^{(1)}$, $i \neq j$; there are 12 generators, corresponding exactly to $\Phi_{SL_4}$. Since $\dim  \overline{W}_{0}^{\theta^{\vee}}=6$, the CDGA $K^{\bullet}(\{\Delta_{i,j}^{(1)}\})$ for $i \neq j$ has cohomology concentrated in degrees $[-6,0].$ We record the graded Poincar\'e polynomial of $H^*(K^{\bullet}(\{\Delta_{i,j}^{(1)}\})$, again in table form. 

\begin{center}
\begin{tabular}{ c c c c c c c | c }
 1& & & & & & & (9)\\
  3& & & & & & & (8)\\
  & 30& & & & & & (7)\\
  & & 62& & & & & (6)\\
  & &4 & 40& & & & (5)\\
   & & &40 & 4& & & (4) \\
  & & & & 62& & & (3) \\
 &  &  &   & & 30 & & (2) \\
 & &  &  & & & 3  & (1) \\
 & & & & & & 1 & (0) \\
\hline
$H^{-6}$  & $H^{-5}$ & $H^{-4}$ & $H^{-3}$ & $H^{-2}$ & $H^{-1}$ & $H^{0}$  \\  
\end{tabular}
\end{center}

\begin{example}
The derived intersection $\mathbb{V}^{\bullet}(I^T(\overline{W}^{2 \alpha_1^{\vee}+2\alpha_2^{\vee}+\alpha_3^{\vee}}_0))$.
\end{example}

Once again, the ideal of the $T$-fixed subscheme is minimally generated by the 12 generators $\Delta_{i,j}^{(1)}$, for $i \neq j$. We have $\dim  \overline{W}_{0}^{2 \alpha_1^{\vee}+2\alpha_2^{\vee}+\alpha_3^{\vee}}=10$ so by Corollary \ref{amplitude} we know the cohomology of the Koszul complex $K^{\bullet}(\{ \Delta_{i,j}^{(1)}\}_{i \neq j})$  is concentrated in degrees $[-2, 0]$. The graded Poincar\'e polynomial is as follows:

\begin{center}
\begin{tabular}{ c c c | c}
1& & & (7) \\
3& & & (6) \\
5 & & & (5) \\
3 &12 & & (4) \\
& 12 &3 & (3) \\
 &  & 5 & (2) \\
 &  & 3 & (1) \\
 & & 1 & (0) \\
 \hline
 $H^{-2}$ & $H^{-1}$ & $H^{0}$ & \\
\end{tabular}
\end{center}

\section{Global theorems in type A}\label{globaltheorems}
In this section we extract global theorems for affine Schubert varieties from the preceding considerations. We continue to keep $G=SL_{n+1}$. Recall that $\varpi_i^{\vee}$ are the fundamental coweights. These coweights are not in the coroot lattice, but $(n+1)k \varpi_1^{\vee}$ and $(n+1)k \varpi_n^{\vee}$ are dominant coweights in $\mathbb{Z} \Phi^{\vee}_{SL_{n+1}}$ for all $k$.

\begin{theorem}\label{lcicriterion}
Let $\overline{\mathrm{Gr}}^{\lambda}$ be an affine Schubert variety for $G=SL_{n+1}$, with $\lambda$ a dominant coweight. Then 
\[(\overline{\mathrm{Gr}}^{\lambda})^T \rightarrow \overline{\mathrm{Gr}}^{\lambda}\] is a local complete intersection (LCI) if and only if $\lambda = (n+1)  k \varpi_{1}^{\vee}$ or $ \lambda = (n+1)k \varpi_{n}^{\vee}$ for $k \in \mathbb{N}$.
\end{theorem}

\begin{proof}
\textit{Step 1:} We first show that if $(\overline{\Gr}^\lambda)^T \to \overline{\Gr}^\lambda$ is a complete intersection, then $\lambda=(n+1)k\varpi_1^{\vee}$ or $ \lambda = (n+1)k \varpi_{n}^{\vee}$ for $k \in \mathbb{N}$.

The case when $n=1$ is trivial, so we assume $n \ge 2$. Let $\lambda=a_1 \alpha_1^\vee+\dots +a_n\alpha_n^\vee$. Assume $a_1 \le a_n$ (otherwise, we can apply the outer-automorphism of $G$ defined by the unique nontrivial Dynkin diagram automorphism). The condition that $\lambda$ is dominant is equivalent to the following $n$ inequalities:
%\begin{equation*}
%	\begin{cases}
%		\langle \lambda,\alpha_1 \rangle=2a_1-a_2 \ge 0;\\
%		\langle \lambda,\alpha_i \rangle=2a_i-a_{i-1}-a_{i+1} \ge 0, \text{ for }1<i<n;\\
%		\langle \lambda,\alpha_1 \rangle=2a_n-a_{n-1} \ge 0.
%	\end{cases}
%\end{equation*}
\begin{equation*}\
	\langle \lambda,\alpha_k \rangle=2a_k-a_{k-1}-a_{k+1} \ge 0, \text{ for }1\le k\le n;
\end{equation*}
here, we set $a_0=a_{n+1}=0$. These inequalities are equivalent to 
\[a_{k+1}-a_k \le a_k-a_{k-1}, \text{ for }1\le k\le n,\]
so
\[a_{k}-a_{k-1} \le a_1-a_0=a_1, \text{ for }1\le k\le n+1.\]
Therefore, 
\begin{equation}\label{eq:ineq1}
	a_k=a_k-a_0=\sum_{i=1}^k a_i-a_{i-1}\le \sum_{i=1}^k a_1=k a_1, \text{ for }1\le k\le n,
\end{equation}
thus
\begin{equation}\label{eq:ineq2}
	\langle\lambda,2\rho\rangle =\sum_{k=1}^n 2a_k \le \sum_{k=1}^n 2k a_1=n(n+1)a_1=|\Phi_{A_{n+1}}|a_1.
\end{equation}
According to Proposition \ref{prop:min ai}, the righthand side of this inequality is the minimum number of generators of the ideal $I^T (\overline{W}^{\lambda}_0)$. We also know the lefthand side of this inequality is the dimension of $\overline{\Gr}^\lambda$. 
If \[(\overline{\Gr}^\lambda)^T \to \overline{\Gr}^\lambda,\] is $LCI$ then so is
\[(\overline{W}^{\lambda}_0)^T \rightarrow \overline{W}^{\lambda}_0.\] If we assume this latter map is $LCI$ then we must have an equality between $\langle \lambda, 2 \rho \rangle$ and the minimal number of generators of $I^T(\overline{W}^{\lambda}_0)= |\Phi_{A_{n}}| a_1$. This forces the inequalities \eqref{eq:ineq1} \eqref{eq:ineq2} to be equalities, thus
$a_k=k a_1, \text{ for }1\le k\le n,$
and $\lambda=a_1 (n+1)\varpi^{\vee}_1$.
If we assume $a_1 >a_n$, the proof proceeds along similar lines and leads to the conclusion that $\lambda= a_n(n+1) \varpi^{\vee}_n$.

%Firstly, we will inductively show that $a_k \ge a_1$ for all $1\le k \le n$. 

%Now we will inductively show that $a_k \le k a_1$ for all $1\le k\le n$. When $k=1$ there is nothing to prove. Assume $a_i \le i a_1$ for $i\le k$. From the $k$-th equation in \eqref{eq:inductive}, we know
%\[a_{k+1}\le 2 a_{k}-a_{k-1} \le 2 k a_1-\]
	\textit{Step 2:} We now show that if $\lambda = k(n+1) \varpi^{\vee}_1$ or $k(n+1) \varpi^{\vee}_n$ for $k \in \mathbb{N}$, then \[(\overline{\mathrm{Gr}}^{\lambda})^T \rightarrow \overline{\mathrm{Gr}}^{\lambda}\] is LCI. We assume $\lambda = k (n+1) \varpi^{\vee}_1$, with $\lambda = k(n+1) \varpi^{\vee}_n$ being proved similarly.
	%Showing $(\overline{\mathrm{Gr}}^{\lambda})^T \rightarrow \overline{\mathrm{Gr}}^{\lambda}$ is LCI is equivalent to showing, for each $t^{\mu} \in \overline{\mathrm{Gr}}^{\lambda}$, that $\mathbb{V}^{\bullet}(I^T(\overline{W}^{\lambda}_{\mu}))$ is classical. Thus, for $\mathfrak{m}$ the maximal ideal of $\mathbb{C}[\overline{W}^{\lambda}_{\mu}]$, what must be shown is that $dim (\mathfrak{m}/\mathfrak{m}^2)^{\neq 0}$= $dim \overline{W}^{\lambda}_{\mu}$. 
	 We again appeal to the work of \cite{KMWY} and \cite{KPZ}: $(T_{\mu} \overline{\mathrm{Gr}}^{\lambda}_{SL_{n+1}})^{\neq 0}$ is spanned by the images of root curves. Thus, there is a tangent vector $X_{\alpha-k\varpi_{r}}$ of $\mathbb{T}$-weight $\alpha-k\varpi_{r}$ in $T_{\mu}\overline{\mathrm{Gr}}^{\lambda}_{SL_{n+1}}$ if and only if $t^{\mu-(\langle \mu, \alpha \rangle +k) \alpha^{\vee}} \in \overline{\mathrm{Gr}}^{\lambda}$, and $t^{\mu-(\langle \mu, \alpha \rangle +k) \alpha^{\vee}} \neq t^{\mu}$, see \cite[Lem. 2.2.2]{ZhuDem}. This is equivalent to $V_{\lambda}(\mu-(\langle \mu, \alpha \rangle+k) \alpha^{\vee}) \neq 0$ and $\langle \mu, \alpha \rangle +k \neq 0$, where $V_{\lambda}$ is the irreducible representation of $PGL_{n+1}$ of highest weight $\lambda$. Thus we see that, up to reindexing $k$, $\dim (T_{\mu}\overline{\mathrm{Gr}}^{\lambda})^{\neq 0}$ is the cardinality of the set $(\alpha^{\vee}, k) \in \Phi^{\vee} \times \mathbb{Z}^{>0}$ such that $V_{\lambda}(\mu+k\alpha^{\vee}) \neq 0$.

For $V_{\lambda}(\mu) \neq 0$ let $RL^{\lambda}(\mu, \alpha^{\vee})= \operatorname{max}\{k| V_{\lambda}(\mu+k\alpha^{\vee}) \neq 0\}$.	Another way of phrasing the above is 
\[\dim (T_{\mu} \overline{\mathrm{Gr}}^{\lambda})^{\neq 0} = \sum_{\alpha^{\vee} \in \Phi^{\vee}_{SL_{n+1}} }RL^{\lambda}(\mu, \alpha^{\vee}).\]

Let $P_{\lambda}$ be the convex hull of $V_{\lambda}$ in $\mathfrak{h}^*_{\mathbb{R}}$.
 Recall \cite{Kam} that the convex hull $P_{\lambda}$ is given by the following inequalities:
$\mu \in P_{\lambda}$ if and only if for all $w \in W$ and all $\varpi_i$ fundamental weights,

\[\langle \mu, w \varpi_i \rangle \geq \langle w_0 \lambda, \varpi_i \rangle.\]

Usually, it is difficult to predict, for $\mu$ integral in $V_{\lambda}$, which ``wall'' of the form 
\[\langle -, w \varpi_i \rangle = \langle w_0 \lambda, \varpi_i \rangle\]
the (co)root string $\mu+k \beta^{\vee}$ will hit, for $\beta^{\vee} \in \Phi^{\vee}$, $k \in \mathbb{Z}^{\geq 0}$; and by difficult to predict, we mean that this depends on $\mu$ inside $V_{\lambda}$. For $V_{\lambda}(\mu)$, where $V_{\lambda}$ has walls of the form $\langle -, w \varpi_i \rangle = \langle w_0 \lambda, \varpi_i \rangle$, the root string associated to $\beta^{\vee}$ and $\mu$ can intersect walls such that $\langle \beta^{\vee}, w \varpi_i \rangle <0$. See Figure \ref{fig1}.

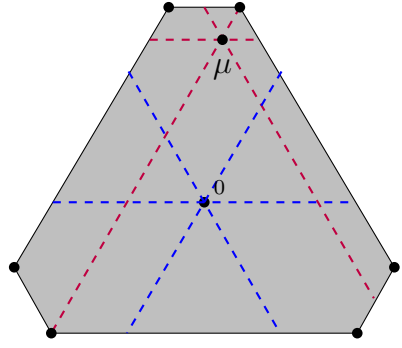
\begin{figure}
\begin{center}
\begin{tikzpicture}[>=triangle 45]

\draw[fill=lightgray] (2.51,-.86) -- (.47,2.58) -- (-.47,2.58) -- (-2.51,-.86) -- (-2.02,-1.73205) -- (2.02,-1.73205) -- (2.51,-.86);
%\draw[thick,red,->] (0,0) -- (1,0);
%\draw[thick,red,->] (0,0) -- (-0.5,0.866025);

\draw[thick,dashed,purple] (0,2.58) -- (2.24,-1.27);
\draw[thick,dashed,purple] (-.72,2.15) -- (.72,2.15);
\draw[thick,dashed,purple] (.47,2.58) -- (-2.02,-1.73);
%\fill (-0.5,-0.866025) circle [radius=0.07];
\fill (0,0) circle [radius=0.07];
\fill (.24,2.15) circle[radius=0.07];

\draw[thick,dashed,blue] (-2,0) -- (2,0);
\draw[thick,dashed,blue] (-1,1.73) -- (1,-1.73);
\draw[thick,dashed,blue] (1.02,1.73) -- (-1.02,-1.73);

\fill (.47,2.58) circle [radius=0.07];
\fill (-.47,2.58) circle [radius=0.07];
\fill (-2.51,-.86) circle [radius=0.07];
\fill (-2.02,-1.73205) circle [radius=0.07];
\fill (2.02,-1.73205) circle [radius=0.07];
\fill (2.51,-.86) circle [radius=0.07];
%\node at (-1.5,-0.866025) {$\times$};
%\node at (1.5,-0.866025) {$\times$};
%\node at (1.2,1.93205) {\small $e$};
%\node at (-1.2,1.93205) {\small $s_1$};
%\node at (1.2,-1.98205) {\small $s_2s_1$};
%\node at (-1.2,-1.98205) {\small $w_0$};
%\node at (2.3,0) {\small $s_2$};
%\node at (-2.5,0) {\small $s_1s_2$};

\node at (0.2,0.2) {\tiny $0$};
\node at (.24,1.76) {$\mu$};

%\node at (-1.75,-1.196025) {$\mu''$};

\end{tikzpicture}
\caption{\label{hey1} \small For $V_{\lambda}$, with $\lambda$ not of the form $(n+1)k \varpi^{\vee}_1$ or $(n+1)k\varpi^{\vee}_n$, (co)root lines through $\mu$ may hit different walls depending on the weight $\mu$ in $V_{\lambda}.$}
\label{fig1}
\end{center}
\end{figure}

The special cases of $P_{(n+1)k \varpi^{\vee}_1}$ (as well as $P_{(n+1)k \varpi^{\vee}_n})$, on the other hand, are very degenerate; the inequalities corresponding to $i, w$ for $i >1$ are redundant. The convex hull $P_{(n+1) k \varpi^{\vee}_1}$ is given by the following inequalities:
$\mu \in P_{(n+1)k \varpi^{\vee}_1}$ if and only if 
\[\langle \mu, \varpi_1 \rangle \geq -k,\]
\[\langle \mu, s_1 \varpi_1 \rangle \geq -k,\]
\[ \langle \mu, s_2 s_1 \varpi_1 \rangle \geq -k,\]
\[\dots \]
\[\langle \mu, s_n s_{n-1} \dots s_1 \varpi_1 \rangle \geq -k.\]

 Consider $V_{(n+1)k \varpi^{\vee}_1}$. Given any $\beta^{\vee} \in \Phi^{\vee}_{SL_{n+1}}$, there exists a unique weight of the form $w \varpi_1$ such that $\langle \beta^{\vee}, w \varpi_1 \rangle <0$. Thus, for $V_{(n+1)k \varpi^{\vee}_1}(\mu) \neq 0$, the ``wall'' hit by the ray  $\mu+k \beta^{\vee}$, $\beta \in \Phi^{\vee}$, $k \in \mathbb{R}^{\geq 0}$ is independent of $\mu$ in $V_{(n+1)k \varpi^{\vee}_1}$. This allows us to obtain a closed expression for all the (co)rootlines through $\mu$ in $V_{(n+1) k \varpi^{\vee}_1}$.

Let us be more specific. We may partition  $\Phi^{\vee}_{SL_{n+1}}$ according to ``column''. The reader may verify that the coroots $-\alpha^{\vee}_1, -\alpha^{\vee}_1-\alpha^{\vee}_2, \dots -\theta^{\vee}$ from the first column all pair negatively with $\varpi_1$. 

Coroots from the second column $\alpha^{\vee}_1, -\alpha^{\vee}_2, -\alpha^{\vee}_2-\alpha^{\vee}_3, \dots$ all pair negatively with $s_1 \varpi_1$.

Coroots from the third column $\alpha^{\vee}_1+\alpha^{\vee}_2, \alpha^{\vee}_2, -\alpha^{\vee}_3, \dots$ all pair negatively with $s_2 s_1 \varpi_1$ and so on and so forth.

For $V_{(n+1)k\varpi^{\vee}_1}(\mu) \neq 0$, and $\beta^{\vee}$ any coroot from the first column, we have  \[RL^{(n+1)k \varpi^{\vee}_1}(\mu, \beta^{\vee})=\langle \mu, \varpi_1\rangle+k.\] 
For any coroot $\gamma^{\vee}$ from the second column, we have 
\[RL^{(n+1)k \varpi^{\vee}_1}(\mu, \gamma^{\vee})= \langle \mu, s_1\varpi_1 \rangle +k.\]
For any coroot $\eta^{\vee}$ from the third column, we have 
\[RL^{(n+1)k \varpi^{\vee}_1}(\mu, \eta^{\vee}) = \langle \mu, s_2 s_1 \varpi_1 \rangle +k,\]
and so on and so forth, \textit{and these expressions hold for any nontrivial $\mu$.}

Summing over all like terms we obtain for any nontrivial weight $\mu$ of $V_{(n+1)k \varpi^{\vee}_1}$,

\begin{align*}
\dim (T_{\mu}\overline{\mathrm{Gr}}^{(n+1)k \varpi_1^{\vee}})^{\neq 0} &= \sum_{\beta^{\vee} \in \Phi^{\vee}_{SL_{n+1}}} RL^{(n+1)k \varpi^{\vee}_1}(\mu, \beta^{\vee}) \\
& = \sum^{n+1}_{i=1} \sum_{ \beta^{\vee} \in \text{ column }i} RL^{(n+1)k \varpi^{\vee}_1}(\mu, \beta^{\vee}) \\
& =\sum_{i=1}^{n+1} \sum_{\beta^{\vee} \in \text{ column }i} \langle \mu, s_i \dots s_1 \varpi_1 \rangle + k \\
&= n(\langle \mu, \varpi_1 \rangle +k)+n(\langle \mu, s_1 \varpi_1 \rangle +k) + \dots + n(\langle \mu , s_n s_{n-1} \dots s_1 \varpi_1 \rangle+k) \\
& = (n+1) \cdot n \cdot k.\\
\end{align*}

In particular this dimension is independent of $\mu$. Moreover 
\[\dim \overline{\mathrm{Gr}}_{SL_{n+1}}^{(n+1)k \varpi^{\vee}_1}= \langle 2 \rho, (n+1)k \varpi^{\vee}_1 \rangle =(n+1) \cdot n \cdot k.\] Thus for $\lambda = (n+1)k\varpi^{\vee}_1$we have  \[(U^{\lambda}_{\mu})^T \rightarrow U^{\lambda}_{\mu}\] is a complete intersection for each $\mu$ and for $\lambda= (n+1)k \varpi^{\vee}_1$,\[(\overline{\mathrm{Gr}}^{\lambda})^T \rightarrow \overline{\mathrm{Gr}}^{\lambda}\] is a complete intersection. The case of $\lambda = (n+1)k \varpi^{\vee}_n$ is proved \textit{mutatis mutandis}.
\end{proof}

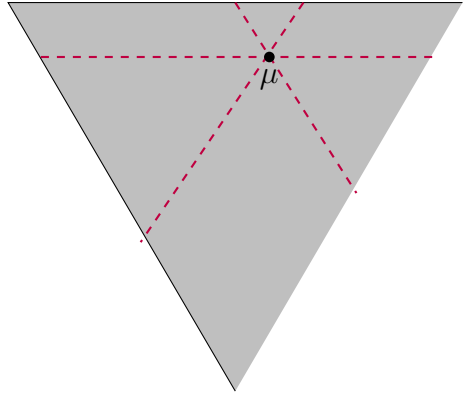
\begin{figure}
\begin{center}
\begin{tikzpicture}[>=triangle 45]

\draw[fill=lightgray] (3,1.72) -- (-3,1.72) -- (0,-3.42);
%\draw[thick,red,->] (0,0) -- (1,0);
%\draw[thick,red,->] (0,0) -- (-0.5,0.866025);

\draw[thick,dashed,purple] (2.6,1) -- (-2.6,1);
\draw[thick,dashed,purple] (0,1.72) -- (1.6,-.8);
\draw[thick,dashed,purple] (.9,1.72) -- (-1.25,-1.45);
%%\fill (-0.5,-0.866025) circle [radius=0.07];
\fill (.45,1) circle [radius=0.07];
\node at (.45,.7) {$\mu$};
%\fill (.24,2.15) circle[radius=0.07];
%
%\draw[thick,dashed,blue] (-2,0) -- (2,0);
%\draw[thick,dashed,blue] (-1,1.73) -- (1,-1.73);
%\draw[thick,dashed,blue] (1.02,1.73) -- (-1.02,-1.73);
%
%\fill (.47,2.58) circle [radius=0.07];
%\fill (-.47,2.58) circle [radius=0.07];
%\fill (-2.51,-.86) circle [radius=0.07];
%\fill (-2.02,-1.73205) circle [radius=0.07];
%\fill (2.02,-1.73205) circle [radius=0.07];
%\fill (2.51,-.86) circle [radius=0.07];
%\node at (-1.5,-0.866025) {$\times$};
%\node at (1.5,-0.866025) {$\times$};
%\node at (1.2,1.93205) {\small $e$};
%\node at (-1.2,1.93205) {\small $s_1$};
%\node at (1.2,-1.98205) {\small $s_2s_1$};
%\node at (-1.2,-1.98205) {\small $w_0$};
%\node at (2.3,0) {\small $s_2$};
%\node at (-2.5,0) {\small $s_1s_2$};

%\node at (0.2,0.2) {\tiny $0$};
%\node at (.24,1.76) {$\mu$};

%\node at (-1.75,-1.196025) {$\mu''$};

\end{tikzpicture}
\caption{\label{hey1} \small For highest weights of the form $k(n+1) \varpi^{\vee}_1$, and $V_{\lambda}(\mu) \neq 0$, the root strings through $\mu$ in direction $-\alpha^{\vee}_1, -\alpha^{\vee}_1-\alpha^{\vee}_2, \dots$ will always hit the wall associated to $\varpi_1$; the rootlines through $\mu$ in direction $\alpha^{\vee}_1$,  $-\alpha^{\vee}_2, \dots$ will always hit the wall associated to $s_1 \varpi_1$, etc, independently of $\mu$. }
\end{center}
\end{figure}

\section{Remarks}\label{remarks}

An important impetus for this work was symplectic duality. Let $\mathcal{M}(\lambda,\mu)$ be the Nakajima quiver variety associated to the pair $(\lambda,\mu)$, and $\mathcal{M}_0 (\lambda,\mu)$ be the spectrum of functions on the Nakajima quiver variety; the natural map
\[\pi:\mathcal{M}(\lambda,\mu) \rightarrow \mathcal{M}_0(\lambda, \mu).\]
is a symplectic resolution. It's expected that $\mathcal{M}_0 (\lambda,\mu)$ serves as the symplectic dual of the affine grassmannian slice $\overline{W}_\mu^\lambda$. The Hikita conjecture in this case predicts an isomorphism of graded rings
\begin{equation}\label{eq:Hikita}
	\C[(\overline{W}_\mu^\lambda)^T]\cong \mathsf{H}^\bullet_{\mathrm{sing}}(\mathcal{M}(\lambda,\mu),\C).
\end{equation}
This isomorphism is proved for $G$ of type $A,D,E_6$ in \cite[Section 8]{KTWWY}.

In our work, $H^0(\mathscr{O}_{\mathbb{V}^{\bullet}(I^T(\overline{W}^{\lambda}_{\mu}))})$ precisely recovers $\mathbb{C}[\overline{W}^{\lambda}_{\mu}]/I^T(\overline{W}_\mu^\lambda)=\C[(\overline{W}_\mu^\lambda)^T]$, and other cohomologies $H^i(\mathscr{O}_{\mathbb{V}^{\bullet}(I^T(\overline{W}^{\lambda}_{\mu}))})$ produce $H^0(\mathscr{O}_{\mathbb{V}^{\bullet}(I^T(\overline{W}^{\lambda}_{\mu}))})=\C[(\overline{W}_\mu^\lambda)^T]$-modules. It would be interesting to construct these modules on the symplectic dual side.

For example, Corollary \ref{pcare} gives an isomorphism of $\C[(\overline{W}_\mu^\lambda)^T]$-modules (up to $\mathbb{T}$-grading shifts)
\begin{equation}\label{eq:dual}
	\left(\C[(\overline{W}_\mu^\lambda)^T]\right)^*
	%(=\mathrm{Hom}_\C (\C[(\overline{W}_\mu^\lambda)^T],\C))
	\cong H^{d-n}(\mathscr{O}_{\mathbb{V}^{\bullet}(I^T(\overline{W}^{\lambda}_{\mu}))}),
\end{equation}
where $d=\dim \overline{W}_\mu^\lambda$, $n=\dim (T_0^* \overline{W}_\mu^\lambda)^{\neq 0}$. On the symplectic dual side, let $\mathcal{L}(\lambda,\mu):=\pi^{-1}(0)$ be the Lagrangian fiber; this is an equidimensional complex projective variety, which is homotopic to $\mathcal{M}(\lambda,\mu)$ in the classical topology (see \cite[Theorem 3.21]{Nak}). Then one can identify
\[(\mathsf{H}^\bullet_{\mathrm{sing}}(\mathcal{M}(\lambda,\mu),\C))^*\cong \mathsf{H}_\bullet^{\mathrm{sing}}(\mathcal{M}(\lambda,\mu),\C)\cong \mathsf{H}_\bullet^{\mathrm{sing}}(\mathcal{L}(\lambda,\mu),\C).\]
 By the Hikita isomorphism \eqref{eq:Hikita} and the isomorphism \eqref{eq:dual}, we can identify 
\[ H^{d-n}(\mathscr{O}_{\mathbb{V}^{\bullet}(I^T(\overline{W}^{\lambda}_{\mu}))}) \cong \mathsf{H}_\bullet^{\mathrm{sing}}(\mathcal{L}(\lambda,\mu),\C)\footnotemark.\]
\footnotetext{We thank Vasily Krylov for his remarks on this point.}

We finally mention one corollary of our duality theorem \ref{thm2} on the symplectic dual side. 
\begin{corollary}
Assume $(\overline{W}^{\lambda}_{\mu})^T \rightarrow \overline{W}^{\lambda}_{\mu}$ is a complete intersection, or equivalently $\dim \overline{W}_\mu^\lambda=\dim(T_0\overline{W}_\mu^\lambda)^{\neq 0}$ by Proposition \ref{prop:basis}; we also assume Hikita conjecture is true for the pair $(\lambda,\mu)$. Then the Lagrangian fiber $\mathcal{L}(\lambda,\mu)$ is irreducible.
\end{corollary}
In this circumstance, by \cite[Theorem 10.2]{Nak}, the $\mu$-weight space in the $\mathfrak{g}^\vee$-highest weight representation $V_{\lambda}$ has dimension $1$.
\begin{proof}
By Corollary \ref{sympoly}, $\dim \overline{W}_\mu^\lambda=\dim(T_0\overline{W}_\mu^\lambda)^{\neq 0}$ implies the Poincare polynomial of $ \C[(\overline{W}_\mu^\lambda)^T]\cong \mathsf{H}^\bullet_{\mathrm{sing}}(\mathcal{M}(\lambda,\mu),\C)=\mathsf{H}^\bullet_{\mathrm{sing}}(\mathcal{L}(\lambda,\mu),\C)$ is palindromic. Then $$\dim \mathsf{H}^{2 \dim \cL(\lambda,\mu)}_{\mathrm{sing}}(\mathcal{L}(\lambda,\mu),\C)=\dim \mathsf{H}^0_{\mathrm{sing}}(\mathcal{L}(\lambda,\mu),\C)=1.$$
Since $\mathcal{L}(\lambda,\mu)$ is equidimensional, we deduce that $\mathcal{L}(\lambda,\mu)$ is irreducible.

\end{proof}

\bibliographystyle{plain}
\bibliography{sn-bibliography}% common bib file

\end{document}